\documentclass[12pt]{article}
\usepackage{setspace}
\usepackage{amsmath,amssymb}
\usepackage{mathtools}
\usepackage{mathrsfs}
\usepackage{geometry}
\usepackage{enumitem}
\usepackage{environ}
\usepackage{setspace}
\usepackage{titlesec}
\usepackage[all,cmtip]{xy}
\usepackage{tikz}
\usetikzlibrary{matrix,arrows}

\usepackage[colorlinks=true,citecolor=blue]{hyperref}
\usepackage[amsthm,thmmarks]{ntheorem}

\NewEnviron{prf}[1][]{\begin{proof}[\bf #1Proof]\BODY\end{proof}}{}
\NewEnviron{slt}[1][]{\begin{proof}[\bf #1	解]\BODY\end{proof}}{}
\newtheorem{definition}{Definition}[section]
\newtheorem{theorem}[definition]{Theorem}
\newtheorem{lemma}[definition]{Lemma}
\newtheorem{proposition}[definition]{Proposition}
\newtheorem{remark}[definition]{Remark}
\newtheorem{example}[definition]{Example}
\newtheorem{corollary}[definition]{Corollary}

\newtheorem{question}[definition]{Question}
\newtheorem{setting}[definition]{Setting}
\newtheorem{condition}[definition]{Condition}
\newtheorem{notation}[definition]{Notation}
\newtheorem{claim}[definition]{Claim}

\newtheorem*{DML}{Dynamical Mordell--Lang Conjecture (DML Conjecture)}

\newcommand\MSC[1]{\textbf{MSC(2020)}: #1}
\newcommand\address[1]{#1}
\newcommand\email[1]{\emph{Email address}: #1}
  
\setlist[enumerate,1]{label=(\roman*)}
\setlist[enumerate,2]{label=(\alph*)}

\title{\textbf{Local height arguments toward the dynamical Mordell--Lang conjecture}}
\author{She Yang\quad and\quad Aoyang Zheng}
\date{}

\begin{document}
\begin{spacing}{1.25}

\maketitle

\begin{abstract}
We consider regular endomorphisms of the complex affine space with a degree gap $k$. They are endomorphisms $f$ of $\mathbb{A}_{\mathbb{C}}^{N}$ of the form $f(x_1,\dots,x_N)=(f_1(x_1,\dots,x_N)+g_1(x_1,\dots,x_N),\dots,\\f_N(x_1,\dots,x_N)+g_N(x_1,\dots,x_N))$, in which $f_1,\dots,f_N$ are homogeneous polynomials of degree $d$ with no nonzero common zeros and $g_1,\dots,g_N$ are polynomials of degree $\leq d-k$. Such an endomorphism extends to an endomorphism of $\mathbb{P}_{\mathbb{C}}^{N}$. Let $H_{\infty}=\mathbb{P}_{\mathbb{C}}^{N}\setminus\mathbb{A}_{\mathbb{C}}^{N}$ be the infinity hyperplane and we denote $f_{\infty}$ as the induced endomorphism of $H_{\infty}$. Suppose that $k$ is twice greater than the multiplicities of $f_{\infty}$ at the periodic closed points, i.e. $k>2\max\limits_{P\in\mathrm{Per}(f_\infty)}e_{f_{\infty}}(P)$. Then we prove that $f$ satisfies the dynamical Mordell--Lang conjecture for curves. As a by-product of our proof, we show that in this case every periodic curve of $f$ is a ``vertical line", i.e. a straight line passing through the origin.

There are many examples which satisfy our condition $k>2\max\limits_{P\in\mathrm{Per}(f_\infty)}e_{f_{\infty}}(P)$. Indeed, we prove that for every $d\geq2$, a general endomorphism $f_{\infty}$ of $H_{\infty}\cong\mathbb{P}_{\mathbb{C}}^{N-1}$ of degree $d$ satisfies $\max\limits_{P\in H_{\infty}(\mathbb{C})}e_{f_{\infty}}(P)\leq(N-1)!\cdot2^{N-1}$. So if we take $k=(N-1)!\cdot2^N+1$, then $f$ will satisfy our condition if $f_{\infty}$ is general (of an arbitrary degree $d\geq k$). Moreover, we provide examples to illustrate that this condition is optimal to force every periodic curve to be a vertical line, in the sense that one cannot change ``$>$" into ``$\geq$".
\end{abstract}


\hspace{5.6em}theorem

\MSC{Primary: 37P05, 37P30; Secondary: 37P15.}

\section{Introduction}

In this article, the base field is $\mathbb{C}$ unless otherwise stated. As a matter of convention, every variety is assumed to be integral, but the closed subvarieties can be reducible. Let $X$ be a variety and let $f$ be an endomorphism of $X$. For a point $x\in X$, the orbit $\mathcal{O}_{f}(x)$ is the set $\{f^{n}(x)\mid n\in\mathbb{N}\}$. We denote $\mathbb{N}=\mathbb{Z}_{+}\cup\{0\}$. An arithmetic progression is a set of the form $\{mk+l\mid k\in\mathbb{N}\}$ for some $m,l\in\mathbb{N}$.

The dynamical Mordell--Lang conjecture is one of the core problems in the field of arithmetic dynamics. It was proposed by Ghioca and Tucker in \cite{GT09} and can be stated as follows.

\begin{DML}
Let $f$ be an endomorphism of a quasi-projective variety $X$. Let $V$ be a closed subvariety of $X$ and let $x\in X(\mathbb{C})$ be a point. Then the return set $\{n\in\mathbb{N}\mid f^n(x)\in V(\mathbb{C})\}$ is a finite union of arithmetic progressions.
\end{DML}

There is an extensive literature on various cases of the DML conjecture. Two significant cases are as follows.

\begin{enumerate}
\item
If $f$ is an \'etale endomorphism of $X$, then the DML conjecture holds. See \cite{Bel06} and \cite[Theorem 1.3]{BGT10}.
\item
If $X = \mathbb{A}^2$, then the DML conjecture holds. See \cite{Xie17} and \cite[Theorem 4]{Xie}.
\end{enumerate}

One can consult \cite{BGT16,Xie} and the references therein for further known results.

In this paper, we mainly study the DML conjecture for curves. So we recall the following notion introduced in \cite[Definition 3]{Xie}.

\begin{definition}
For a quasi-projective variety $X$ and an endomorphism $f$ of $X$, we say $(X,f)$ satisfies the \emph{DML(1) property}, if for any closed subcurve $C\subseteq X$ and any point $x\in X(\mathbb{C})$, the return set $\{n\in \mathbb{N}|\ f^n(x)\in C(\mathbb{C})\}$ is a finite union of arithmetic progressions.
\end{definition}

Here ``1" stands for the dimension of the closed subvariety.

\subsection{Main results}

We will investigate \emph{regular endomorphisms} of $\mathbb{A}^N$ in this paper. We firstly introduce some notions and settings, which will be used throughout the article.

\begin{setting}\label{setting}
A \emph{regular endomorphism} of $\mathbb{A}^N$ is a dominant endomorphism $f$ that can be extended to an endomorphism of $\mathbb{P}^N$. Here, we fix the standard coordinate system $[x_0,\dots,x_N]$ on $\mathbb{P}^N$ and naturally regard $\mathbb{A}^N$ as the open subset $\{x_0\neq0\}$ in it. Therefore, the regular endomorphism $f$ can be written as
$$
f(x_1,\dots,x_N)=(f_1(x_1,\dots,x_N)+g_1(x_1,\dots,x_N),\dots,f_N(x_1,\dots,x_N)+g_N(x_1,\dots,x_N)),
$$
in which $f_1(x_1,\dots,x_N),\dots,f_N(x_1,\dots,x_N)$ are homogeneous polynomials of degree $d\geq1$ that have no nonzero common zeros and $g_1(x_1,\dots,x_N),\dots,g_N(x_1,\dots,x_N)$ are polynomials of degree $<d$.

For $1\leq k\leq d$, we say $f$ has \emph{degree gap} $k$ if the degrees of $g_1(x_1,\dots,x_N),\dots,g_N(x_1,\dots,x_N)$ are not greater than $d-k$.

We denote $H_{\infty}$ as the infinity hyperplane $\{x_0=0\}$ in $\mathbb{P}^N$. We sometimes regard $f$ as an endomorphism of $\mathbb{P}^N$ by abusing notation, and we denote $f_{\infty}=f|_{H_{\infty}}$.

We write $O$ as the origin $(0,\dots,0)\in\mathbb{A}^N$ and let $\pi:\mathbb{P}^N\setminus\{O\}\rightarrow H_{\infty}$ be the central projection. A \emph{vertical line} in $\mathbb{A}^N$ (or $\mathbb{P}^N$) means a straight line passing through $O$. Here ``vertical" means vertical to $H_{\infty}$.

For an irreducible closed subcurve $C\subseteq\mathbb{A}^N$, we denote $\overline{C}$ as its closure in $\mathbb{P}^N$.
\end{setting}

The concept of regular endomorphisms was introduced in \cite{BJ00} and had been studied in several previous works. For example, please see \cite[Subsections 1.4--1.5]{Xie24}, \cite{DFR}, \cite[Subsection 1.2]{Zhong}, and \cite[Subsection 1.2]{JXZ}.

Now we state our first main theorem regarding curves which have an infinite intersection with an orbit of a regular endomorphism. Throughout this paper, for a finite morphism $f:X\rightarrow Y$ between schemes and a point $x\in X$, we let the \emph{multiplicity} of $f$ at $x$ be the length of the Artin local ring $\mathcal{O}_{X,x}/f^*\mathfrak{m}_{f(x)}\mathcal{O}_{X,x}$ and denote this quantity as $e_f(x)$. Here $\mathfrak{m}_{f(x)}$ is the maximal ideal of $\mathcal{O}_{Y,f(x)}$.

\begin{theorem}\label{main}
Let $f$ be a regular endomorphism of $\mathbb{A}^N$ which has degree gap $k$. Let $C\subseteq\mathbb{A}^N$ be an irreducible closed subcurve and let $\overline{C}$ be its closure in $\mathbb{P}^N$. Suppose that there is a point $x\in\mathbb{A}^N(\mathbb{C})$ such that the orbit $\mathcal{O}_f(x)$ has an infinite intersection with $C$. Let $I=\overline{C}\cap H_{\infty}$.
\begin{enumerate}
\item
We have $1\leq|I|\leq2$, and the elements of $I$ are $f_{\infty}$-periodic closed points.

\item
Let $I_0$ be the finite set $\bigcup\limits_{P\in I}\mathcal{O}_{f_{\infty}}(P)$. Suppose $k>2e_{f_{\infty}}(P)$ holds for every $P\in I_0$. Then $C$ is an $f$-periodic vertical line.
\end{enumerate}
\end{theorem}

The following corollary is an immediate consequence of Theorem \ref{main}, which provides an explicit example.

\begin{corollary}\label{cor main}
Let $d>1$ and let $f(x_1,\dots,x_N)=(x_1^d+g_1(x_1,\dots,x_N),\dots,x_N^d+g_N(x_1,\dots,x_N))$ be a regular endomorphism of $\mathbb{A}^N$, in which the degrees of $g_1(x_1,\dots,x_N),\dots,g_N(x_1,\dots,x_N)$ are not greater than $d-3$. Let $C\subseteq\mathbb{A}^N$ be an irreducible closed subcurve. Suppose that there is a point $x\in\mathbb{A}^N(\mathbb{C})$ such that the orbit $\mathcal{O}_f(x)$ has an infinite intersection with $C$, and $\overline{C}\cap H_{\infty}$ is contained in the open set $\{x_1\cdots x_N\neq0\}$. Then $C$ is an $f$-periodic vertical line. In particular, every irreducible $f$-periodic curve $C$ which satisfies $\overline{C}\cap H_{\infty}\subseteq\{x_1\cdots x_N\neq0\}$ is a vertical line.
\end{corollary}

We will provide examples to show that the requirements in Theorem \ref{main} and Corollary \ref{cor main} about the degree gap are optimal to force the periodic curves to be vertical. See Example \ref{example optimal}(i).

We use Theorem \ref{main} to study the dynamical Mordell--Lang conjecture for curves. According to the statement of Theorem \ref{main}, we would like to address the following condition on our regular endomorphism $f$. We use notions introduced in Setting \ref{setting}. We denote $\mathrm{Per}(f_{\infty})$ as the set of $f_{\infty}$-periodic points in $H_{\infty}(\mathbb{C})$.

\begin{condition}\label{conditon k}
Let $f$ be a regular endomorphism of $\mathbb{A}^N$ which has degree gap $k$. Assume $k>2\max\limits_{P\in\mathrm{Per}(f_{\infty})}e_{f_{\infty}}(P)$.
\end{condition}

This condition has a similar flavor to the assumptions made in previous works about regular endomorphisms. We make restrictions on the ramification status of $f_{\infty}$. For instance, both \cite[Theorem 1.20]{Xie24} and \cite[Theorem A]{DFR} consider regular endomorphisms of $\mathbb{A}^2$. Roughly speaking, the former asks the endomorphism $f_{\infty}$ of $\mathbb{P}^1$ to have no exceptional points and the latter requires it to be unramified at periodic points. With this more delicate condition, we can gain some higher dimensional results.

The following proposition guarantees that there are many examples satisfying Condition \ref{conditon k}. Indeed, if we take $k=(N-1)!\cdot2^N+1$, then $f$ will satisfy our condition if $f_{\infty}$ is a general endomorphism of $H_{\infty}\cong\mathbb{P}^{N-1}$ of an arbitrary degree $d\geq k$. We find this proposition interesting on its own. The proof is given in subsection 2.2.

\begin{proposition}\label{bound e}
Let $f$ be a general endomorphism of $\mathbb{P}^N$ of degree $d$. Then $\max\limits_{x\in\mathbb{P}^N(\mathbb{C})}e_{f}(x)\leq N!\cdot2^N$, regardless of what value $d$ takes.
\end{proposition}

Now we can state our second main theorem.

\begin{theorem}\label{main2}
Let $f$ be a regular endomorphism of $\mathbb{A}^N$ satisfying Condition \ref{conditon k}. Then $f$ satisfies the DML(1) property. Moreover, every irreducible $f$-periodic curve is a vertical line.
\end{theorem}

A notable feature of our approach towards the dynamical Mordell--Lang conjecture in this case is that we can determine all periodic curves of the dynamical system, as a by-product of the proof. In particular, for a regular endomorphism
$$
f(x_1,\dots,x_N)=(f_1(x_1,\dots,x_N)+g_1(x_1,\dots,x_N),\dots,f_N(x_1,\dots,x_N)+g_N(x_1,\dots,x_N))
$$
presented as in Setting \ref{setting}, the following statements hold provided that $f$ satisfies Condition \ref{conditon k}. They are easy consequences of Theorem \ref{main2}.
\begin{enumerate}
\item
If $g_1(x_1,\dots,x_N)=\cdots=g_N(x_1,\dots,x_N)=0$, then irreducible $f$-periodic curves are the vertical lines passing through points in $\mathrm{Per}(f_{\infty})$. Here we abuse some notations as the curves are contained in $\mathbb{A}^N$. Notice that under this situation, the degree gap of the iteration $f^n$ is $d^n$ where $d$ is the algebraic degree of $f$. Hence by doing iteration, one can see that this assertion strengthens the ``moreover" part of \cite[Theorem 1.20]{Xie24} and generalizes it to the higher dimensional case.

\item
If the union of $f$-periodic curves (i.e. periodic vertical lines) is Zariski dense in $\mathbb{A}^N$, then we must have $g_1(x_1,\dots,x_N)=\cdots=g_N(x_1,\dots,x_N)=0$. This partially generalizes the first part of \cite[Theorem 1.20]{Xie24}.

\item
In particular, if $g_{1,0}=g_1(0,\dots,0),\dots,g_{N,0}=g_N(0,\dots,0)$ are not all zero, then we may focus on the vertical line passing through $[g_{1,0},\dots,g_{N,0}]\in H_{\infty}(\mathbb{C})$ (we abuse the notation as above). If this line is $f$-invariant, then it is the only irreducible $f$-periodic curve; otherwise, there are no irreducible $f$-periodic curves.
\end{enumerate}

In \cite{MS14}, the authors characterized the invariant subvarieties of split endomorphisms of $(\mathbb{A}^1)^N$. The core part is to characterize the invariant curves of split endomorphisms of $(\mathbb{A}^1)^2$, which is already quite hard. This reflects that it may be unlikely to determine all invariant curves for an arbitrary regular endomorphism of $\mathbb{A}^N$.

The following example illustrates that Condition \ref{conditon k} is optimal to force every periodic curve to be a vertical line, in the sense that one cannot change ``$>$" into ``$\geq$".

\begin{example} \label{example optimal}
\begin{enumerate}
\item
Let $C\subseteq\mathbb{A}^N$ be the irreducible closed subcurve defined by the equations $x_1^2-x_2^2=4$ and $x_2=\cdots=x_N$. Let $T_d(x)$ be the Chebyshev polynomial, so that $T_d(t+\frac{1}{t})=t^d+\frac{1}{t^d}$. Suppose $d>1$ is odd. Then $\widetilde{T_d}(x)\coloneqq i^{-d}T_d(ix)$ satisfies $\widetilde{T_d}(t-\frac{1}{t})=t^d-\frac{1}{t^d}$. Let $f$ be the regular endomorphism of $\mathbb{A}^N$ defined by
$$
f(x_1,x_2,\dots,x_N)=(T_d(x_1),\widetilde{T_d}(x_2),\dots,\widetilde{T_d}(x_N)).
$$
Then $f$ has degree gap $2$ and $C$ is $f$-invariant, while $f_{\infty}$ is unramified at the points in $\overline{C}\cap H_{\infty}=\{[0,\pm1,1,\dots,1]\}$ and $C$ is not a vertical line. This shows the optimality in the settings of Theorem \ref{main} and Corollary \ref{cor main}.

\item
This is a heuristic example. Let $C\subseteq\mathbb{A}^N$ be the irreducible closed subcurve defined by the equations $x_1x_2=1$ and $x_3=\cdots=x_N=0$. Let $d\geq2t$ be positive integers. Consider regular endomorphisms of the form
$$
f(x_1,\dots,x_N)=
\left(
\begin{array}{ccc}
c_1x_1^d+a_{1,d-t}x_1^{d-t}x_2^{t}+\cdots+a_{1,t}x_1^{t}x_2^{d-t}+g_1(x_1,\dots,x_N)+(\deg\leq d-2t) & \\
a_{2,d-t}x_1^{d-t}x_2^{t}+\cdots+a_{2,t}x_1^{t}x_2^{d-t}+c_2x_2^{d}+g_2(x_1,\dots,x_N)+(\deg\leq d-2t) & \\
a_{3,d-t}x_1^{d-t}x_2^{t}+\cdots+a_{3,t}x_1^{t}x_2^{d-t}+g_3(x_1,\dots,x_N)+(\deg\leq d-2t) & \\
\vdots & \\
a_{N,d-t}x_1^{d-t}x_2^{t}+\cdots+a_{N,t}x_1^{t}x_2^{d-t}+g_N(x_1,\dots,x_N)+(\deg\leq d-2t)
\end{array}
\right)
$$
where $c_1c_2=1$ and $g_1,\dots,g_N$ are homogeneous polynomials of degree $d$ such that
$$
g_i(x_1,x_2,0,\dots,0)=0,\quad i=1,\dots,N
$$
One can modify the tails (i.e. lower order terms) to make $C$ into an $f$-invariant curve. Then $f$ has degree gap $k=2t$.

Let $f_{\infty}$ be the homogeneous part and we regard it as an endomorphism of $H_{\infty}\cong\mathbb{P}^{N-1}$. Suppose $t\gg N$, e.g. $t\geq(N-1)!\cdot2^{N-1}$. We suspect that for general choices of the coefficients $c_1,c_2,a_{ij}$ and the polynomials $g_k$, one then get $\max\limits_{P\in H_{\infty}(\mathbb{C})}e_{f_{\infty}}(P)=e_{f_{\infty}}([1,0,0,\dots,0])=e_{f_{\infty}}([0,1,0,\dots,0])=t$. If so, then $k=2\max\limits_{P\in H_{\infty}(\mathbb{C})}e_{f_{\infty}}(P)$ while $C$ is not a vertical line. Hence we get many optimal examples for Theorem \ref{main2}.

\item
For $N=2$, we make the heuristic example in (ii) rigorous. Let
$$
g(x_1,x_2)=
\left(
\begin{array}{ccc}
c_1x_1^d+a_{1,d-t}x_1^{d-t}x_2^{t}+\cdots+a_{1,t}x_1^{t}x_2^{d-t}+(\deg\leq d-2t) & \\
a_{2,d-t}x_1^{d-t}x_2^{t}+\cdots+a_{2,t}x_1^{t}x_2^{d-t}+c_2x_2^{d}+(\deg\leq d-2t)
\end{array}
\right)
$$
be a regular endomorphism of the form as above, where $a_{1,t}a_{2,d-t}\neq0$. Then we have $e_{g_{\infty}}([1,0])=e_{g_{\infty}}([0,1])=t$. For $M$ sufficiently large, we let $\phi_M$ be the $M$-th power map on $\mathbb{A}^2$ and denote $f=g\circ\phi_M$. Then $f$ has degree gap $k=2Mt$ while $\max\limits_{P\in H_{\infty}(\mathbb{C})}e_{f_{\infty}}(P)=Mt$. 
\end{enumerate}
\end{example}

\begin{remark}
\begin{enumerate}
\item
Consider the optimal examples that $C\subseteq\mathbb{A}^N$ is an $f$-invariant irreducible closed subcurve which is not a vertical line, while the degree gap $k=2\max\limits_{P\in\overline{C}\cap H_{\infty}}e_{f_{\infty}}(P)$. All of the examples above satisfy $\deg(C)=2$, and we believe that this is always the case. Moreover, in some sense we believe that Example \ref{example optimal}(ii) gives \emph{all} such optimal examples.

\item
If one wants to see examples with larger $\deg(C)$, then we cannot make $k=2\max\limits_{P\in\overline{C}\cap H_{\infty}}e_{f_{\infty}}(P)$. But by mimicking the example above and ``moving the central point", we can construct examples that $\deg(C)$ is arbitrary large while the quotient of $k$ by $\max\limits_{P\in H_{\infty}(\mathbb{C})}e_{f_{\infty}}(P)$ tends to $2$. We leave the construction to the interested readers.
\end{enumerate}
\end{remark}

\subsection{Discussions}

We make some discussions about the conditions of our main results, and describe the proof strategy.

The idea of using heights to deal with the DML Conjecture dates back to almost 20 years ago. For example, see \cite{BGKT10} and \cite[Theorem 8.1]{GTZ12}. Philosophically, this work is a successor of \cite{XY}, in which the authors use a height argument to investigate the dynamical Mordell--Lang conjecture. As the authors mentioned in the introduction of \cite{XY}, the key point of such an argument is to find \emph{two different speeds of growth}. We recall a main result in \cite{XY} to further illustrate this issue.

\begin{theorem}(\cite[Theorem 1.1]{XY})\label{XY}
Let $X$ be a projective variety over an algebraically closed field of arbitrary characteristic and let $f$ be a surjective endomorphism of $X$. Suppose that any cohomological Lyapunov multiplier of any iterate of $f$ is not a positive integer. Let $C\subseteq X$ be a closed subcurve and let $x\in X(K)$ be a point. If $\overline{\mathcal{O}_f(x)}=X$, then $\mathcal{O}_{f}(x)\cap C(K)$ is a finite set.
\end{theorem}

We will not recall the definition of cohomological Lyapunov multipliers here (see \cite{XY}). Instead, we remark that philosophically the condition of Theorem \ref{XY} means that the dynamical system is ``at the opposite of polarized systems". Hence the system is not ``isotropic", and we are managed to find two speeds of height growth in there. However, the dynamical systems studied in this paper are polarized. Thus, the ``height argument" fails and we need to use a ``local height argument" instead. Also, this explains the philosophy behind Condition \ref{conditon k} --- ``break down the isotropy". To this end, we naturally consider a regular endomorphism $f$ which has a visible degree gap, while the behavior of $f_{\infty}$ is general.

The article \cite{XYZ} is a baby forerunner of this paper. The main result that the authors proved in \cite{XYZ} can be summarized in the following proposition. By combining with some previous results, the authors solved the dynamical Mordell--Lang conjecture for split self-maps of affine curve times projective curve in \cite{XYZ}. The base field of \cite{XYZ} is $\overline{\mathbb{Q}}$.

\begin{proposition}\label{xyz}
Let $f$ be an endomorphism of $\mathbb{A}^1$ and $g$ be an endomorphism of $\mathbb{P}^1$. Suppose $\mathrm{deg}(f)=\mathrm{deg}(g)\geq 2$ and $g$ has no exceptional point. Let $C\subseteq\mathbb{A}^1\times\mathbb{P}^1$ be an irreducible closed subcurve which is not a fiber of $\mathbb{A}^1$ or $\mathbb{P}^1$. Then for every $x_0\in\mathbb{A}^1(\overline{\mathbb{Q}})$ and $y_0\in\mathbb{P}^1(\overline{\mathbb{Q}})$, the set $C\cap\mathcal{O}_{f\times g}(x_0,y_0)$ is finite.
\end{proposition}

The proof of Proposition \ref{xyz} is by using a local height argument, which we briefly recall as follows. Let $\overline{C}$ be the closure of $C$ in $\mathbb{P}^1\times\mathbb{P}^1$ and let $L_{\infty}$ be the infinity line in $\mathbb{P}^1\times\mathbb{P}^1$. Then for an appropriate place, the points in $\mathcal{O}_{f\times g}(x_0,y_0)$ tend to $L_{\infty}$ with the maximal speed. So a subsequence of the points in $C\cap\mathcal{O}_{f\times g}(x_0,y_0)$ shall tend to a point in $P\in\overline{C}\cap L_{\infty}$, for which the $\mathbb{P}^1$-component of the points must also tend to $P$ on $L_{\infty}$ with the maximal speed. However, Silverman's result \cite[Theorem E]{Sil93} guarantees that this cannot happen.

One can see that the key point in the procedure above is finding two different speeds of local height growth at a certain place. Then Silverman's result controls the speed of the slower one. By looking into the proof of \cite[Theorem E]{Sil93}, one finds that Roth's theorem lies at the core. As a result, Roth's theorem will also play a central role in this article. Matsuzawa has used Roth's theorem to study problems in high-dimensional arithmetic dynamics \cite{Mat23,Mat25}. Inspired by his work, we try to use the local height argument to study the dynamical Mordell--Lang conjecture.

Now we briefly introduce the proof strategy of Theorem \ref{main}. As we work over an arbitrary field of characteristic 0 instead of number fields, we need to use the Arakelov theory of arithmetic function fields (see subsection 2.1). Roth's theorem over arithmetic function fields is established by Vojta \cite{Voj21}. This setting leads to some extra technical subtleties, in comparison to the setting over number fields (see the $L^1$-integrability condition in Definition \ref{3.1}).

We find an arithmetic function field $K$ (i.e. a finitely generated extension field of $\mathbb{Q}$) such that all data are defined over $K$. The proof can be divided into three steps.

\textbf{Step 1:} We find a sequence of irreducible closed subcurves
$$
\cdots\stackrel{f}\rightarrow C_{-n}\stackrel{f}\rightarrow\cdots\stackrel{f}\rightarrow C_{-1}\stackrel{f}\rightarrow C_0=C
$$
of $\mathbb{A}_K^N$ such that each $C_{-n}$ has an infinite intersection with the orbit $\mathcal{O}_f(x)$. Then for every $n$, we have $f(\overline{C_{-n-1}}\cap H_{\infty})=\overline{C_{-n}}\cap H_{\infty}$ and $1\leq|\overline{C_{-n}}\cap H_{\infty}|\leq2$. The latter is guaranteed by Siegel's theorem \ref{siegel}. Then Northcott's finite property \ref{northcott} forces all points in $\overline{C_{-n}}\cap H_{\infty}$ to be $f_{\infty}$-periodic. This finishes the proof of Theorem \ref{main}(i). Moreover, with the notation as in Theorem \ref{main}, we can see that $I_0=\bigcup\limits_{n\in\mathbb{N}}\left(\overline{C_{-n}}\cap H_{\infty}\right)$. This argument is learned from \cite[Section 8]{Xie17}.

\textbf{Step 2:} By adapting an argument using triangle inequality, we find that the curves become more and more ``vertical" at the (one or two) points in $\overline{C_{-n}}\cap H_{\infty}$. See Proposition \ref{3.15}.

\textbf{Step 3:} Using Roth's theorem, we conclude that $C_{-n}$ must be a vertical line if it is sufficiently vertical at the points in $\overline{C_{-n}}\cap H_{\infty}$. See Proposition \ref{4.5}. This finishes the proof.

One can see that the usage of Roth's theorem in our work is not quite the same as that in \cite[Theorem E]{Sil93} or \cite{Mat23,Mat25}. In previous works, Roth's theorem is often adapted to the setting ``if $f(x)$ is sufficiently near $P$, then $x$ must be very close to a point in $f^{-1}(P)$". In our work, Roth's theorem is used to prove that a ``sufficiently vertical" curve must be a vertical line.

~

At the end of the Introduction, we outline the structure of this paper. In Section 2, we recall the Arakelov theory of arithmetic function fields and prove Proposition \ref{bound e}. In Section 3, we introduce the notion of local heights and establish some inequalities for future proof. We place some detailed proofs of technical lemmas in the Appendix. Then in Section 4, we prove our main result and its consequences, i.e. Theorem \ref{main}, Corollary \ref{cor main}, and Theorem \ref{main2}. Finally, in Section 5, we consider the positive characteristic case of our setting and make some discussions.

\textbf{Statement on A.I. use.} Example \ref{AIexample} is learned from ChatGPT 6 Astra. We did not use A.I. in the main part of this article.

\section{Preparations}

In subsection 2.1, we recall some knowledge of Arakelov theory for arithmetic function fields, which will be the fundamental setting of our work. We will recall Roth's theorem, Siegel's theorem, and the Northcott property in this setting.

In subsection 2.2, we prove Proposition \ref{bound e}. This statement does not lie in the mainstream of this article, but it is also interesting on its own. We find it an appropriate place to give its proof in here.

\subsection{Arakelov theory for arithmetic function fields}

In this subsection, we closely follow \cite[Sections 3--4]{Voj21} to recall the Arakelov theory for arithmetic function fields. All results in this subsection can be found in loc. cit.. We will use \cite{Voj21} as a main reference for the Arakelov theory through this work. One can look into it for the definitions of terminologies in Arakelov theory (such as arithmetic varieties, smoothly metrized line sheaves, etc.).

An \emph{arithmetic function field} $K$ is by definition a finitely generated extension field of $\mathbb{Q}$. Let $d=\mathrm{tr.deg}_{\mathbb{Q}}(K)$. Let $B$ be a normal arithmetic variety, whose function field is isomorphic to $K$. Let $\mathcal{M}$ be an ample smoothly metrized line sheaf on $B$. Then we get a big polarization $M=(B;\mathcal{M})$. This is an abbreviation of $(B;\mathcal{M},\dots,\mathcal{M})$ in which there are $d$ copies of $\mathcal{M}$.

The polarization $M$ gives a set of absolute values $M_{K}=M_{K}^{\infty}\sqcup M_{K}^{0}$ of $K$ in the following way.

Let $B^{(1)}$ be the set of prime Weil divisors on $B$. The set $M_{K}^{0}$ of non-archimedean absolute values will have a bijection with $B^{(1)}$. For every $Y\in B^{(1)}$, let $h_M(Y)=c_1(\mathcal{M}|_{Y})^d$ be the arithmetic intersection number. As $h_M(Y)\geq0$, the formula $\|x\|_Y=\mathrm{e}^{-h_M(Y)\mathrm{ord}_Y(x)}$ for $x\in K^{\times}$ defines a non-archimedean absolute value of $K$. This finishes the description of $M_K^{0}$.

The set $M_{K}^{\infty}$ of archimedean absolute values will have a bijection with $B(\mathbb{C})^{\mathrm{gen}}:=B(\mathbb{C})\setminus\bigcup\limits_{Y\in B^{(1)}}Y(\mathbb{C})$. For every $b\in B(\mathbb{C})^{\mathrm{gen}}$, the formula $\|x\|_b=|x(b)|$ for $x\in K$ defines an archimedean absolute value of $K$. This finishes the description of $M_K^{\infty}$.

Now we make $M_K$ into a measure space. Let $\mu_{\mathrm{fin}}$ be the counting measure on $M_K^0$, and let $\mu_{\infty}$ be the Lebesgue measure on $B(\mathbb{C})$ associated to the semipositive $(d,d)$-form $c_1(\|\cdot\|_{\mathcal{M}})^d$. Combining $\mu_{\mathrm{fin}}$ and $\mu_{\infty}$, we get a measure $\mu$ on $B(\mathbb{C})\sqcup M_K^0\supseteq M_K$. Notice that $B(\mathbb{C})\setminus M_{K}^{\infty}$ has measure zero and $M_K^{\infty}$ has finite measure.

In this setting, we have a \emph{product formula} which says that $\int_{M_K}\log\|x\|_vd\mu(v)=0$ for every $x\in K^{\times}$. See \cite[Section 3.2]{Mor00}.

The product formula allows us to define a ``\emph{naive height}" on $\mathbb{P}^N(K)$.

\begin{definition}\label{2.1}
The \emph{naive height} function $h_K:\mathbb{P}^N(K)\rightarrow\mathbb{R}_{\geq0}$ is defined by $h_K([x_0,\dots,x_N])=\int_{M_K}\log\max\{\|x_0\|_v,\dots,\|x_N\|_v\}d\mu(v)$.
\end{definition}

In particular, by taking $N=1$, we get a naive height function $h_K:K\rightarrow\mathbb{R}_{\geq0}$ given by $h_K(x)=\int_{M_K}\log^{+}\|x\|_vd\mu(v)$. Here and in the sequel, we denote $\log^{+}x=\max\{0,\log x\}$ and $\log^{-}x=\min\{0,\log x\}$ for $x\in\mathbb{R}_{\geq0}\cup\{+\infty\}$. So $\log^{+}x$ takes value in $\mathbb{R}_{\geq0}\cup\{+\infty\}$ and $\log^{-}x$ takes value in $\mathbb{R}_{\leq0}\cup\{-\infty\}$.

Using Arakelov theory, we can extend the naive height function to algebraic points as follows. See \cite[Section 3.3]{Mor00} and also \cite[Subsection 3C]{Voj21}.

Let $\mathcal{O}(1)_{\mathrm{can}}$ be the continuously metrized line sheaf on $\mathbb{P}_{\mathbb{Z}}^N$ equipped with the canonical metric. Pulling it back to $\mathbb{P}_{B}^N$, we get a continuously metrized line sheaf on $\mathbb{P}_{B}^N$. By abusing notation, we still call this object as $\mathcal{O}(1)_{\mathrm{can}}$. Since $\mathbb{P}_K^N$ is the generic fiber of the projection $\mathbb{P}_{B}^N\rightarrow B$, we can talk about the closure $\bar{x}$ in $\mathbb{P}_{B}^N$ of a closed point $x\in\mathbb{P}_K^N$.

\begin{definition}\label{2.2}
Let $\pi$ be the projection map $\mathbb{P}_{B}^N\rightarrow B$. For $x\in\mathbb{P}^N(\overline{K})$, we define $h_{\mathcal{O}(1)_{\mathrm{can}}}(x)=\frac{c_1(\pi^*\mathcal{M}|_{\bar{x}})^d\cdot c_1(\mathcal{O}(1)_{\mathrm{can}}|_{\bar{x}})}{[\kappa(x):K]}$. Here $\kappa(x)$ is the residue field of the closed point $x\in\mathbb{P}_K^{N}$.
\end{definition}

According to \cite[Proposition 3.3.2]{Mor00}, we have $h_K(x)=h_{\mathcal{O}(1)_{\mathrm{can}}}(x)$ for every $x\in\mathbb{P}^N(K)$. By base change \cite[Proposition 3.3.1]{Mor00}, we can see that the function $h_{\mathcal{O}(1)_{\mathrm{can}}}$ also takes value in $\mathbb{R}_{\geq0}$.

The following theorem is the main result in Moriwaki's height machinery.

\begin{theorem}(Northcott's finiteness property; \cite[Theorem 4.3]{Mor00})\label{northcott}
For all $C>0$ and $n\in\mathbb{Z}_+$, the set $\{x\in\mathbb{P}^N(\overline{K})|\ h_{\mathcal{O}(1)_{\mathrm{can}}}(x)\leq C\text{ and }[\kappa(x):K]\leq n\}$ is finite.
\end{theorem}

Next, we state Roth's theorem over arithmetic function fields. Let $S_0$ be a finite subset of $M_K^0$ and let $S=M_K^{\infty}\sqcup S_0\subseteq M_K$.

\begin{theorem}(Roth's theorem; \cite[Theorem 4.4]{Voj21}) \label{roth}
Let $a_1,\dots,a_q$ be distinct elements of $K$; let $\varepsilon>0$; and let $c\in\mathbb{R}$. Then the inequality
$$
\int_{S}\left(\sum\limits_{j=1}^{q}-\log^{-}\|x-a_j\|_v\right)d\mu(v)\leq(2+\varepsilon)h_K(x)+c
$$
holds for all but finitely many $x\in K$.
\end{theorem}

Finally, we state Siegel's theorem on integral points on curves due to Lang. The statement as follows is an immediate corollary of \cite[Corollary 4.11]{Voj21}.

\begin{theorem}\label{siegel}
Let $R\subseteq K$ be a finitely generated subring over $\mathbb{Z}$. Let $C\subseteq\mathbb{A}_K^N$ be an irreducible closed subcurve. Suppose $C(K)\cap R^N$ is infinite. Then with the notation as in Setting \ref{setting}, the set $(\overline{C}\cap H_{\infty})(\overline{K})$ contains at most two points.
\end{theorem}

With the hypothesis as above, Siegel's theorem also asserts that $C$ is a rational curve. For our purpose, we pay more attention to the ``at most two points at infinity" part of the theorem.

\subsection{Proof of Proposition \ref{bound e}}

The first-named author asked a question about whether ``the ramifications of a general endomorphism of the projective space should be mild" on MathOverflow \cite{LittleBear}, and the second-named author posted the accepted answer on the webpage. In this subsection, we arrange the answer into a proof of Proposition \ref{bound e}.

Firstly, let us clarify the meaning of a ``general endomorphism". Let $d$ and $N$ be positive integers. Let $\mathrm{Hom}_d(\mathbb{P}^N)$ be the set of endomorphisms of $\mathbb{P}^N$ of algebraic degree $d$. We know that such an endomorphism is determined by $N+1$ homogeneous polynomials of degree $d$, in which there are $N'=(N+1)\binom{N+d}{d}$ coefficients in total. Thus we get a set injection $\mathrm{Hom}_d(\mathbb{P}^N)\hookrightarrow\mathbb{P}^{N'-1}(\mathbb{C})$. It is a fact that the image is open in Zariski topology, and hence in this way we can regard $\mathrm{Hom}_d(\mathbb{P}^N)$ as a quasi-projective variety. We say \emph{a general endomorphism in} $\mathrm{Hom}_d(\mathbb{P}^N)$ \emph{has a property} $\mathcal{P}$ if there exists an open dense subset $U\subseteq\mathrm{Hom}_d(\mathbb{P}^N)$ such that every $f\in U(\mathbb{C})$ satisfies $\mathcal{P}$.

For an endomorphism $f\in\mathrm{Hom}_d(\mathbb{P}^N)(\mathbb{C})$, we denote $e(f)=\max\limits_{x\in\mathbb{P}^N(\mathbb{C})}e_f(x)$. The definition of the multiplicities $e_f(x)$ can be found before Theorem \ref{main}. We have $1\leq e(f)\leq d^N$ for every $f$. We prove that the function $e$ is well-behaved.

\begin{proposition}\label{2.6}
The function $f\mapsto e(f)$ on $\mathrm{Hom}_d(\mathbb{P}^N)(\mathbb{C})$ is upper semicontinuous with respect to the Zariski topology.
\end{proposition}

Before starting the proof, we recall the following result due to \cite[Corollary 4.8]{LJT74}. See also \cite[Lemma C.2]{Mat25}.

\begin{lemma}\label{2.7}
Let $f:X\rightarrow Y$ be a finite flat morphism of varieties. Then the function $x\mapsto e_f(x)$ is upper semicontinuous on $X$.
\end{lemma}

\proof[Proof of Proposition \ref{2.6}]
We write $B=\mathrm{Hom}_d(\mathbb{P}^N)$ for simplicity. We consider the morphism $F:\mathbb{P}^N\times B\rightarrow\mathbb{P}^N\times B$ defined by the requirement $F(x,f)=(f(x),f)$ on closed points. One can check that $F$ is indeed a morphism of varieties. We see that $F$ is a $B$-morphism. For every $f\in B(\mathbb{C})$, the map $F_f$ on the fiber is just the endomorphism $f$ of $\mathbb{P}^N$. Thus $F$ is quasi-finite and proper, hence finite. It is also flat as $\mathbb{P}^N\times B$ is smooth. Moreover, for every closed point $(x,f)$ in $\mathbb{P}^N\times B$, we can see that $e_F((x,f))=e_f(x)$ by base change to the closed fiber.

For every integer $M$, we denote $Y_M=\{y\in\mathbb{P}^N\times B|\ e_F(y)\geq M\}$. Lemma \ref{2.7} guarantees that $Y_M$ is a closed subset of $\mathbb{P}^N\times B$. Hence the image $\mathrm{pr}_B(Y_M)$ under the projection $\mathrm{pr}_B:\mathbb{P}^N\times B\rightarrow B$ is a closed subset of $B$. But we can see that for an endomorphism $f\in B(\mathbb{C})$, the condition $e(f)\geq M$ is equivalent to saying that $f\in\mathrm{pr}_B(Y_M)(\mathbb{C})$. Hence the result follows.
\endproof

Proposition \ref{2.6} guarantees that a general $f\in\mathrm{Hom}_d(\mathbb{P}^N)(\mathbb{C})$ satisfies $e(f)=\min\limits_{f\in\mathrm{Hom}_d(\mathbb{P}^N)(\mathbb{C})}e(f)$. So our task becomes to find a special $f$ which has a relatively small $e(f)$.

\begin{lemma}\label{2.8}
There exists an endomorphism $f\in\mathrm{Hom}_d(\mathbb{P}^N)(\mathbb{C})$ such that $e(f)\leq N!\cdot2^N$.
\end{lemma}

\begin{proof}
The assertion is easy to be verified when $N=1$. Indeed, certain Latt\`es maps satisfy the requirement. Let $g$ be such an endomorphism of $\mathbb{P}^1$ of degree $d$. We have $e(g)=2$. Since $\mathbb{P}^N$ is the quotient of $(\mathbb{P}^1)^N$ by the symmetric group $S_N$ which acts as permuting the components, we can get an endomorphism $f\in\mathrm{Hom}_d(\mathbb{P}^N)(\mathbb{C})$ induced by the endomorphism $g^N$ of $(\mathbb{P}^1)^N$. More precisely, we get such an endomorphism $f$ that satisfies $\pi\circ g^N=f\circ\pi$, where $\pi$ denotes the quotient map $(\mathbb{P}^1)^N\rightarrow\mathbb{P}^N$. Therefore, we have
\begin{align*}
e(f)
&\leq\max\limits_{y\in(\mathbb{P}^1)^N(\mathbb{C})}e_{f\circ\pi}(y)=\max\limits_{y\in(\mathbb{P}^1)^N(\mathbb{C})}e_{\pi\circ g^N}(y)\leq\max\limits_{y\in(\mathbb{P}^1)^N(\mathbb{C})}e_{g^N}(y)\cdot\max\limits_{y\in(\mathbb{P}^1)^N(\mathbb{C})}e_{\pi}(y) \\
&\leq e(g)^N\cdot\mathrm{deg}(\pi)=N!\cdot2^N.
\end{align*}
Here we use the formula ``$e_{g\circ f}(x)=e_f(x)e_g(f(x))$" for finite flat morphisms in \cite[Lemma C.3]{Mat25}.
\end{proof}

\proof[Proof of Proposition \ref{bound e}]
Combining Proposition \ref{2.6} and Lemma \ref{2.8}, we get the assertion immediately.
\endproof

\begin{remark}
\begin{enumerate}
\item
It is also interesting to consider the behavior of the function $f\mapsto\max\limits_{x\in\mathrm{Per}(f)}e_f(x)$ on $\mathrm{Hom}_d(\mathbb{P}^N)(\mathbb{C})$, which is closer to our Condition \ref{conditon k}. However, this function seems to have a worse behavior. Indeed, we believe that $\max\limits_{x\in\mathrm{Per}(f)}e_f(x)=1$ for a very general $f$, and one cannot change ``very general" into ``general". Hence this function should not be upper semicontinuous.

\item
Our bound $N!\cdot2^N$ in Proposition \ref{bound e} should be far from optimal. We suspect that the optimal bound is $N+1$ for small $N$, while in general the growth of the bound may be at least as $2^{\sqrt{N}}$. This is because every endomorphism of $\mathbb{P}^N$ seems to have a point at which the rank of the Jacobian matrix is at most $N-\lfloor\sqrt{N}\rfloor$. See the discussion in \cite{LittleBear}. It is interesting to calculate the actual optimal bound for small $N$.
\end{enumerate}
\end{remark}

\section{Some inequalities of local heights}

In this section, we introduce a basic notion of local heights (i.e. Weil functions) on $\mathbb{P}^N$. Over number fields, the theory of local heights is standard, which measures the degree of closeness of a point with a fixed closed subscheme in a projective variety. For example, see \cite[Chapter 10]{Lan83} and \cite{Sil87}. In this paper, for our purpose, we need this notion for rational points in $\mathbb{P}^N$ as well as the infinity hyperplane in $\mathbb{P}^N$ over arithmetic function fields. In the setting of arithmetic function fields, the theory of local heights for Cartier divisors in projective varieties was developed in \cite[Subsection 3D]{Voj21}. But it seems that the theory for closed subschemes of higher codimension has not been developed yet. We will also not try to develop a general theory in this article. Instead, we will only give some definitions and prove some inequalities that will be used in the sequel. We restrict ourselves in $\mathbb{P}^N$ and work with its natural coordinates in order to maintain an elementary and concrete flavor of our article.

We inherit the setting from subsection 2.1. Let $K$ be an arithmetic function field and let $M=(B;\mathcal{M})$ be a big polarization of $K$. Then we get a measure space $M_K$ of absolute values on $K$.

From now on, we \emph{fix} a finite set $S_0\subseteq M_K^0$, and denote $S=M_K^{\infty}\sqcup S_0$. Notice that $S$ has finite measure. We will work with this fixed set of places $S$ in the following of this section. As a result, according to the notion of \cite[Definition 3.22]{Voj21}, our ``Weil functions" in this section are indeed ``partial Weil functions over $S$". But we still call them Weil functions for short.

The concept of $S$-\emph{constants} will play a key role in the definition of Weil functions. The following definition is a variant of \cite[Definition 3.18]{Voj21}.

\begin{definition}\label{3.1}
An $S$-\emph{constant} is a measurable, $L^1$ function from $S$ to $\mathbb{R}_{\geq0}$. An $S$-constant is usually denoted as $(c_v)_{v\in S}$.
\end{definition}

\begin{remark}\label{3.2}
The sum and maximum of two $S$-constants is an $S$-constant. As mentioned in \cite[Remark 3.20]{Voj21}, the function $(|\log\|a\|_v|)_{v\in S}$ is an $S$-constant for all $a\in K^{\times}$. Therefore, for every $a_1,\dots,a_n\in K$ and every $c\in\mathbb{R}_{\geq0}$, the function $(\log^{+}\|a_1\|_v+\cdots+\log^{+}\|a_n\|_v+c)_{v\in S}$ is also an $S$-constant as $S$ has finite measure. Indeed, every $S$-constant in the sequel will essentially have this form. We often implicitly use this remark to guarantee that a certain formula indeed defines an $S$-constant.
\end{remark}

Now we start to define the Weil functions in $\mathbb{P}^N$ for rational points and the infinity hyperplane. The base field will always be $K$ in this section. We rely on the standard coordinate system $[x_0,\dots,x_N]$ on $\mathbb{P}^N$.

We inherit the notation in Setting \ref{setting}. Namely, we write $\mathbb{A}^N$ as the open subset $\{x_0\neq0\}$ and write $H_{\infty}$ as the closed subset $\{x_0=0\}$ of $\mathbb{P}^N$, respectively.

We remark that although some first properties in the following will be easy to experts, we choose to give detailed proofs of almost all the statements. This is because the $L^1$-integrable condition in the definition of $S$-constants deserves some attention. See Appendix for the proofs of the lemmas in this section.

We begin with the definition of Weil functions for rational points. Let $P=[a_0,\dots,a_N]\in\mathbb{P}^N(K)$.

For $i\in\{0,\dots,N\}$, if the coordinate $a_i\neq0$, then we define
$$
\lambda_P^{i}(v,x)=\left\{
\begin{array}{cc}
-\log^{-}\max\limits_{0\leq j\leq N,j\neq i}\|\frac{x_j}{x_i}-\frac{a_j}{a_i}\|_v, & x_i\neq0 \\
0, & x_i=0
\end{array}
\right.
$$
for $v\in S$ and $x=[x_0,\dots,x_N]\in\mathbb{P}^N(K)$. So $\lambda_P^i$ is a map from $S\times\mathbb{P}^N(K)$ to $\mathbb{R}_{\geq0}\cup\{+\infty\}$ which takes value $+\infty$ if and only if $x=P$. Notice that we define $\lambda_P^i$ only when the coordinate $a_i\neq0$.

We prove that the function $\lambda_P^i$ is essentially independent of the choice of $i$, in the sense that two different choices can bound each other by an $S$-constant. In the following, for a subset $X\subseteq\mathbb{P}^N(K)$ (which will be either $\mathbb{A}^N(K)$ or $\mathbb{P}^N(K)$ in practice), we say that two functions $\lambda_1,\lambda_2:S\times X\rightarrow\mathbb{R}_{\geq0}\cup\{+\infty\}$ \emph{can bound each other by an} $S$-\emph{constant} if there is an $S$-constant $(c_v)_{v\in S}$ such that $\lambda_1(v,x)\leq\lambda_2(v,x)+c_v$ and $\lambda_2(v,x)\leq\lambda_1(v,x)+c_v$ hold for every $v\in S$ and every $x\in X$.

\begin{lemma}\label{3.3}
Let $P=[a_0,\dots,a_N]\in\mathbb{P}^N(K)$. Suppose that $a_ia_j\neq0$ for two subscripts $i,j$. Then the functions $\lambda_P^i,\lambda_P^j:S\times\mathbb{P}^N(K)\rightarrow\mathbb{R}_{\geq0}\cup\{+\infty\}$ can bound each other by an $S$-constant.
\end{lemma}

The lemma guarantees that we can make the following definition.

\begin{definition}
Let $P=[a_0,\dots,a_N]\in\mathbb{P}^N(K)$. Suppose $a_i\neq0$. A \emph{Weil function for} $P$ is a function $\lambda_P:S\times\mathbb{P}^N(K)\rightarrow\mathbb{R}_{\geq0}\cup\{+\infty\}$ which can bound each other by an $S$-constant with the function $\lambda_P^i$ defined as above.
\end{definition}

\begin{remark}
Lemma \ref{3.3} guarantees that the definition is independent of the choice of $i$. Also, since an $S$-constant (as a function from $S$ to $\mathbb{R}_{\geq0}$) cannot take value $\infty$, we see that $\lambda_P(v,x)=+\infty$ if and only if $x=P$.
\end{remark}

We fix some notations, including the definition of the \emph{Weil function for} $H_{\infty}$.

\begin{notation}
\begin{enumerate}
\item
For $P=[a_0,\dots,a_N]\in\mathbb{P}^N(K)$, the notation $\lambda_P:S\times\mathbb{P}^N(K)\rightarrow\mathbb{R}_{\geq0}\cup\{+\infty\}$ will always stand for a Weil function for $P$.

\item
The \emph{Weil function for} $H_{\infty}$ is the function $\lambda_{H_{\infty}}:S\times\mathbb{A}^N(K)\rightarrow\mathbb{R}_{\geq0}$ defined by $\lambda_{H_{\infty}}(v,x)=\log^{+}\max\limits_{1\leq i\leq N}\|x_i\|_v$ for every $v\in S$ and every $x=(x_1,\dots,x_N)\in\mathbb{A}^N(K)$.
\end{enumerate}
\end{notation}

Now we establish some first properties of the Weil functions.

\begin{lemma}\label{3.7}
Let $A$ be an automorphism of $\mathbb{P}^N$ and let $P\in\mathbb{P}^N(K)$. Then the functions $\lambda_P(v,x)$ and $\lambda_{A(P)}(v,A(x))$ from $S\times\mathbb{P}^N(K)$ to $\mathbb{R}_{\geq0}\cup\{+\infty\}$ can bound each other by an $S$-constant.
\end{lemma}

\begin{lemma}\label{3.8}
Let $f$ be an endomorphism of $\mathbb{P}^N$ and let $P\in\mathbb{P}^N(K)$.
\begin{enumerate}
\item
There is an $S$-constant $(c_v)_{v\in S}$ such that $\lambda_P(v,x)\leq\lambda_{f(P)}(v,f(x))+c_v$ holds for every $v\in S$ and every $x\in\mathbb{P}^N(K)$.

\item
There exist $S$-constants $(c_{1,v})_{v\in S}$ and $(c_{2,v})_{v\in S}$ such that for every $v\in S$ and every $x\in\mathbb{P}^N(K)$, we have $\lambda_{f(P)}(v,f(x))\leq e_f(P)\lambda_P(v,x)+c_{2,v}$ when $\lambda_P(v,x)\geq c_{1,v}$. Here $e_f(P)$ denotes the multiplicity of $f$ at $P$.
\end{enumerate}
\end{lemma}

Recall that an endomorphism of $\mathbb{A}^N$ is said to be regular if it can be extended to an endomorphism of $\mathbb{P}^N$. See Setting \ref{setting}.

\begin{lemma}\label{3.9}
Let $f$ be a regular endomorphism of $\mathbb{A}^N$ of algebraic degree $d$. Then the functions $\lambda_{H_{\infty}}(v,f(x))$ and $d\lambda_{H_{\infty}}(v,x)$ from $S\times\mathbb{A}^N(K)$ to $\mathbb{R}_{\geq0}$ can bound each other by an $S$-constant.
\end{lemma}

We also need a statement which says that a point can be separated from a closed subset by an $S$-constant.

\begin{lemma}\label{3.10}
Let $V$ be a closed subvariety of $\mathbb{P}^N$ and let $P\in(\mathbb{P}^N\setminus V)(K)$. Then there is an $S$-constant $(c_v)_{v\in S}$ such that for every $v\in S$ and every $x\in\mathbb{P}^N(K)$, we have $x\notin V(K)$ when $\lambda_P(v,x)\geq c_{v}$.
\end{lemma}

Next, we let $C\subseteq\mathbb{A}^N$ be an irreducible closed subcurve. We consider the behavior of Weil functions near points in $(\overline{C}\cap H_{\infty})(K)$, where $\overline{C}$ denotes the closure of $C$ in $\mathbb{P}^N$.

\begin{lemma}\label{3.11}
Let $C\subseteq\mathbb{A}^N$ be an irreducible closed subcurve and let $P\in(\overline{C}\cap H_{\infty})(K)$. Denote $i(P)$ as the intersection number of $\overline{C}$ and $H_{\infty}$ at $P$. Then there exist $S$-constants $(c_{1,v})_{v\in S}$ and $(c_{2,v})_{v\in S}$ such that for every $v\in S$ and every $x\in C(K)$, we have $\lambda_{H_{\infty}}(v,x)\leq i(P)\lambda_P(v,x)+c_{2,v}$ when $\lambda_P(v,x)\geq c_{1,v}$.
\end{lemma}

We introduce the following notations to simplify the writing of proofs.

\begin{notation}\label{3.12}
For a polynomial $f(x_1,\dots,x_n)=\sum\limits_{i_1,\dots,i_n\geq0}a_{i_1,\dots,i_n}x_1^{i_1}\cdots x_n^{i_n}$, we denote $N(f)=|\{(i_1,\dots,i_n)\in\mathbb{N}^n|\ a_{i_1,\dots,i_n}\neq0\}|$. For a place $v\in S$, we denote $H_v(f)=\max\limits_{(i_1,\dots,i_n)\in\mathbb{N}^n}\|a_{i_1,\dots,i_n}\|_v$ and $l_v(f)=\sum\limits_{(i_1,\dots,i_n)\in\mathbb{N}^n}\log^+\|a_{i_1,\dots,i_n}\|_v$.
\end{notation}

We make some remarks before stating the next proposition, which amounts to saying that a point in $C$ which is close to $H_{\infty}$ must be close to a point in $\overline{C}\cap H_{\infty}$. In comparison with the previous statements, the proof of this proposition is somehow more difficult and non-standard. This is because this proposition is a priori global in nature, while all the previous statements are essentially local. In the setting of number fields, there is only a finite number of places under consideration, and then one can reduce to a local problem by using compactness. However, this approach is non-effective. So in the setting of arithmetic function fields, we need to work harder in order to gain the $L^1$-integrability of the $S$-constants. The method of the proof is due to pinaki \cite{pinaki}.

\begin{proposition}\label{3.13}
Let $C\subseteq\mathbb{A}^N$ be an irreducible closed subcurve. Assume $\overline{C}\cap H_{\infty}$ contains only $K$-rational points, and denote them as $P_1,\dots,P_t\in H_{\infty}(K)$. Then there is an $S$-constant $(c_v)_{v\in S}$ such that $\lambda_{H_{\infty}}(v,x)\leq\sum\limits_{j=1}^{t}i(P_j)\lambda_{P_j}(v,x)+c_v$ holds for every $v\in S$ and every $x\in C(K)$.
\end{proposition}

\begin{proof}
Lemma \ref{3.7} and Lemma \ref{3.9} guarantee that implementing a linear transformation of $\mathbb{A}^N$ will not affect the conclusion. Therefore, we may assume that $P_1,\dots,P_t$ all lie in the open set $\{x_1\cdots x_N\neq0\}$.

For $k\in\{1,\dots,N\}$, we introduce a finite set $Y_k\subseteq H_{\infty}(K)$ as follows. We write $P_j=[0,b_{1,j},\dots,b_{k-1,j},1,b_{k+1,j},\dots,b_{N,j}]$ for every $1\leq j\leq t$. We define
$$
Y_k=\{[0,b_1,\dots,b_{k-1},1,b_{k+1},\dots,b_N]\in H_{\infty}(K)|\ b_l\in\{b_{l,1},\dots,b_{l,t}\},\forall l\in\{1,\dots,k-1,k+1,\dots,N\}\}.
$$

\begin{claim}\label{3.14}
For every $k\in\{1,\dots,N\}$, there exist an $S$-constant $(c_{k,v})_{v\in S}$ and a number $M_k>0$ such that $\lambda_{H_{\infty}}(v,x)\leq M_k\sum\limits_{Q\in Y_k}\lambda_{Q}(v,x)+c_{k,v}$ holds for every $v\in S$ and every $x=[1,x_1,\dots,x_N]\in C(K)$ which satisfies $\|x_k\|_v=\max\limits_{1\leq i\leq N}\|x_i\|_v$.
\end{claim}

\begin{proof}
We shall only prove the claim for $k=1$ as the other cases are similar. Our goal is to find $M>0$ and an $S$-constant $(c_v)_{v\in S}$ such that $\lambda_{H_{\infty}}(v,x)\leq M\sum\limits_{Q\in Y_1}\lambda_{Q}^1(v,x)+c_{v}$ holds for every $v\in S$ and every $x=[1,x_1,\dots,x_N]\in C(K)$ which satisfies $\|x_1\|_v=\max\limits_{1\leq i\leq N}\|x_i\|_v$.

~
        
First we prove the case when $N=2$. In this case, we have $Y_1=\{P_1,\dots,P_t\}$. The curve $C\subseteq\mathbb{A}^2$ is defined by a single equation $g(x_1,x_2)=0$. Write $g(x_1,x_2)=g_0(x_1,x_2)+h(x_1,x_2)$, where $g_0$ is homogeneous of $\deg g$ and $\deg h<\deg g$.
        
Let $c_v=l_v(h)+\log ^+N(h)$ and $M=\max\limits_{1\leq j\leq t}i(P_j)$. We prove $\lambda_{H_\infty}(v,x)\leq M\sum\limits_{1\leq j \leq t} \lambda_{P_j}^1(v,x)+c_v$ for every $v\in S$ and every $x=[1,x_1,x_2]\in C(K)$ which satisfies $\|x_1\|_v\geq \|x_2\|_v$.
        
Write $P_j=[0,1,b_j]$ for short. Then we may write $g_0(x_1,x_2)=\prod\limits_{1\leq j\leq t}(x_2-b_jx_1)^{i(P_j)}$. We assume $\lambda_{H_\infty}(v,x)>0$ so that $\|x_1\|_v>1$. We further assume $\lambda_{H_\infty}(v,x)>\lambda_{P_j}^1(v,x)$ for $1\leq j\leq t$. Then $\lambda_{P_j}^1(v,x)=-\log^-\|\frac{x_2}{x_1}-b_j\|_v$. Since
$$
\prod\limits_{1\leq j\leq t}\|\frac{x_2}{x_1}-b_j\|_v^{i(P_j)}=\|\frac{1}{x_1}\|_v\cdot \|\frac{h(x_1,x_2)}{x_1^{\deg g-1}}\|_v,
$$
we have $\lambda_{H_\infty}(v,x)=\log \|x_1\|_v\leq M\sum\limits_{1\leq j\leq t}\lambda_{P_j}^1(v,x)+\log^+(N(h)H_v(h))$. This finishes the proof for $N=2$.

~

For the general case, we reduce to the case when $N=2$. For $2\leq i\leq N$, let $\pi_i:\mathbb{P}^N\dashrightarrow \mathbb{P}^2$ be the projection $\pi_i([x_0,\dots,x_N])=[x_0,x_1,x_i]$, and $L_{\infty}=\{x_0=0\}\subseteq\mathbb{P}^2$. By the assumption that every $P_j$ lies in $\{x_1\cdots x_N\neq 0\}$, we know $\overline{C}$ does not intersect the indeterminacy locus of each $\pi_i$. Now the morphism $\pi_i|_{\overline{C}}:\overline{C}\rightarrow\pi_i(\overline{C})$ is surjective. Also, a point in $\overline{C}$ maps into $L_\infty$ if and only if it lies in $H_\infty$. Hence $\pi_i(\overline{C})$ is an irreducible closed subcurve of $\mathbb{P}^2$ and $\pi_i(\overline{C})\cap L_{\infty}=\{\pi_i(P_j)|\ 1\leq j\leq t\}$. In particular, $\pi_i(\overline{C})\cap L_{\infty}$ contains only rational points. 
        
Thus we can apply the $N=2$ case. For every $2\leq i\leq N$, there exist a positive integer $M_i$ and an $S$-constant $(c_{i,v})_{v\in S}$ such that $\lambda_{L_{\infty}}(v,\pi_i(x))\leq M_i\sum\limits_{1\leq j \leq t}\lambda_{\pi_i(P_j)}^1(v,\pi_i(x))+c_{i,v}$ holds for every $v\in S$ and every $x\in C(K)$ which satisfies $\|x_1\|_v=\max\limits_{1\leq i\leq N}\|x_i\|_v$.

Let $c_v=\max\limits_{2\leq i\leq N}c_{i,v}$ and let $M=t\max\limits_{2\leq i\leq N}M_i$. In order to prove the expected statement, we may assume $\lambda_{H_\infty}(v,x)>c_v\geq0$. Then $\lambda_{H_\infty}(v,x)=\lambda_{L_\infty}(v,\pi_i(x))=\log\|x_1\|_v$ for every $2\leq i\leq N$.

For every $2\leq i\leq N$, we have
\begin{align*}
\log\|x_1\|_v&\leq M_i\sum\limits_{1\leq j \leq t}\lambda_{\pi_i(P_j)}^1(v,\pi_i(x))+c_{i,v}\\
&\leq M_it\max\limits_{1\leq j\leq t}\lambda_{\pi_i(P_j)}^1(v,\pi_i(x))+c_{i,v}\\
&\leq M_it\max\limits_{1\leq j\leq t}-\log^-\|\frac{x_i}{x_1}-b_{i,j}\|_v+c_{i,v}.
\end{align*}
Hence we have
$$
\max\limits_{1\leq j\leq t}-\log\|\frac{x_i}{x_1}-b_{i,j}\|_v\geq\frac{\log\|x_1\|_v-c_{i,v}}{M_it}
$$
for every $2\leq i\leq N$. Therefore, we get
$$
\max\limits_{Q\in Y_1}\lambda_{Q}^1(v,x)\geq\frac{\log\|x_1\|_v-c_v}{M}=\frac{\lambda_{H_\infty}(v,x)-c_v}{M}
$$
and finish the proof.
\end{proof}

Now we prove the proposition. Firstly, for every $1\leq k\leq N$, we may enlarge the $S$-constant in Claim \ref{3.14} and find $(c_{k,v}')_{v\in S}$ such that $\lambda_{H_{\infty}}(v,x)\leq M_k\sum\limits_{j=1}^t\lambda_{P_j}(v,x)+c_{k,v}'$ holds for every $v\in S$ and every $x=[1,x_1,\dots,x_N]\in C(K)$ which satisfies $\|x_k\|_v=\max\limits_{1\leq i\leq N}\|x_i\|_v$. We may also only do this for $k=1$. Indeed, Lemma \ref{3.10} guarantees that there is an $S$-constant $(c_{1,v}'')_{v\in S}$ such that $x\notin C(K)$ if $\lambda_{Q}(v,x)\geq c_{1,v}''$ for some $Q\in Y_1\setminus \{P_1,\dots,P_t\}$. Therefore, letting $c_{1,v}'=c_{1,v}+|Y_1|M_kc_{1,v}''$ suffices.
    
Next, we let $M=\max\limits_{1\leq k\leq N}M_k$ and $c_v'=\max\limits_{1\leq k\leq N}c_{k,v}'$. Then $\lambda_{H_{\infty}}(v,x)\leq M\sum\limits_{j=1}^t\lambda_{P_j}(v,x)+c_{v}'$ holds for every $v\in S$ and every $x=[1,x_1,\dots,x_N]\in C(K)$.

According to Lemma \ref{3.11}, for every $1\leq j\leq t$, there are $S$-constants $(c_{j,1,v})_{v\in S}$ and $(c_{j,2,v})_{v\in S}$ such that for every $v\in S$ and every $x\in C(K)$, we have $\lambda_{H_{\infty}}(v,x)\leq i(P_j)\lambda_{P_j}(v,x)+c_{j,2,v}$ when $\lambda_{P_j}(v,x)\geq c_{j,1,v}$. We let $c_v=\sum\limits_{j=1}^{t}\left(Mc_{j,1,v}+c_{j,2,v}\right)+c_v'$. Then one can verify that $\lambda_{H_{\infty}}(v,x)\leq\sum\limits_{j=1}^{t}i(P_j)\lambda_{P_j}(v,x)+c_v$ holds for every $v\in S$ and every $x\in C(K)$. Thus we finish the proof.
\end{proof}

Finally, we need the following statement proved by triangle inequality. Same as in Setting \ref{setting}, we let $O$ be the origin $(0,\dots,0)\in\mathbb{A}^N(K)$ and let $\pi:\mathbb{P}^N\setminus\{O\}\rightarrow H_{\infty}$ be the central projection. Also, please see Setting \ref{setting} for the definitions of the ``degree gap" of a regular endomorphism and the notation $f_{\infty}$.

\begin{proposition}\label{3.15}
Let $f$ be a regular endomorphism of $\mathbb{A}^N$ of algebraic degree $d$ with degree gap $k$. Let $P\in H_{\infty}(K)$ and let $X\subseteq\mathbb{A}^N(K)$ be a subset.

Let $r$ be a positive real number. Suppose there are $S$-constants $(c_{1,v})_{v\in S}$ and $(c_{2,v})_{v\in S}$ such that for every $v\in S$ and every $x\in X$, we have $f(x)\neq O$ and
$$
r\lambda_{H_{\infty}}(v,f(x))\leq\lambda_{f_{\infty}(P)}(v,\pi(f(x)))+c_{2,v}
$$
when $\lambda_{f_{\infty}(P)}(v,f(x))\geq c_{1,v}$.

Let $r'=\frac{\min\{k,dr\}}{e_{f_{\infty}}(P)}$ in which $e_{f_{\infty}}(P)$ denotes the multiplicity of $f_{\infty}$ at $P$. Then there exist $S$-constants $(c_{3,v})_{v\in S}$ and $(c_{4,v})_{v\in S}$ such that for every $v\in S$ and every $x\in X$, we have $x\neq O$ and
$$
r'\lambda_{H_{\infty}}(v,x)\leq\lambda_{P}(v,\pi(x))+c_{4,v}
$$
when $\lambda_{P}(v,x)\geq c_{3,v}$.
\end{proposition}

\begin{proof}
Notice that linear transformations of $\mathbb{A}^N$ commute with $\pi$ and do not change $d,k$ and $e_{f_{\infty}}(P)$. Lemma \ref{3.7} and Lemma \ref{3.9} guarantee that implementing such a transformation will not affect the conclusion. Therefore, we may assume that $f_{\infty}(P)=P=[0,1,0,\dots,0]$. Write
$$
f=[x_0^d,f_1(x_1,\dots,x_N)+x_0^kg_1(x_0,\dots,x_N),\dots,f_N(x_1,\dots,x_N)+x_0^kg_N(x_0,\dots,x_N)]
$$
and $x=[1,x_1,\dots,x_N]=(x_1,\dots,x_N)$. Write $f(x)=[1,x_1',\dots,x_N']$ for short.

We modify the $S$-constants $(c_{1,v})_{v\in S}$ and $(c_{2,v})_{v\in S}$ such that for every $v\in S$ and every $x\in X$, we have $f(x)\neq O$ and
$$
r\lambda_{H_{\infty}}(v,f(x))\leq\lambda_{f_{\infty}(P)}^1(v,\pi(f(x)))+c_{2,v}
$$
when $\lambda_{f_{\infty}(P)}^1(v,f(x))\geq c_{1,v}$.

Let $(c_{5,v})_{v\in S}$ be the $S$-constant as in Lemma \ref{3.8}(i) such that $\lambda_P^1(v,x)\leq \lambda_{f_\infty(P)}^1(v,f(x))+c_{5,v}$. Let $c_{6,v}=\sum\limits_{1\leq i\leq N} \left(l_v(g_i) +\log^+ N(g_i)\right)$ for $v\in S$. Let $(c_{7,v})_{v\in S}$ be the $S$-constant as in Lemma \ref{3.9} such that $\lambda_{H_\infty}(v,f(x))+c_{7,v}\geq d\lambda_{H_\infty}(v,x)$. 
    
Let $(c_{8,v})_{v\in S}$ and $(c_{9,v})_{v\in S}$ be the $S$-constants as in Lemma \ref{3.8}(ii) such that for every $v\in S$ and every $y\in H_{\infty}(K)$, we have
$$
e_{f_\infty}(P)\lambda_P^1(v,y)+c_{9,v}\geq \lambda_{f_\infty(P)}^1(v,f_\infty(y))
$$
when $\lambda_P^1(v,y)\geq c_{8,v}$.

Let $c_{3,v}=c_{1,v}+c_{5,v}+c_{8,v}+1$ and $c_{4,v}=\frac{1}{e_{f_\infty}(P)}\left(c_{2,v}+c_{6,v}+(1+r)c_{7,v}+c_{9,v}+2\log 2\right)$. We prove that for every $v\in S$ and every $x\in X$, we have $x\neq O$ and
$$
r'\lambda_{H_{\infty}}(v,x)\leq\lambda_{P}^1(v,\pi(x))+c_{4,v}
$$
when $\lambda_{P}^1(v,x)\geq c_{3,v}$.

Firstly, we have $\lambda_{H_\infty}(v,x)=\log\|x_1\|_v>0$ under this assumption as $c_{3,v}>0$. In particular $x\neq O$. Also, we have $\lambda_{f_\infty(P)}^1(v,f(x))\geq \lambda_P^1(v,x)-c_{5,v}\geq c_{3,v}-c_{5,v}>0$. Thus $\|x_1'\|_v>1$ and $\|\frac{x_i'}{x_1'}\|_v<1$ for $2\leq i\leq N$. Then we get $\lambda_{H_\infty}(v,f(x))=\log \|x_1'\|_v$.
    
Again, we have $\lambda_{f_\infty(P)}^1(v,f_\infty(\pi(x)))\geq \lambda_{P}^1(v,\pi(x))-c_{5,v}\geq \lambda_P^1(v,x)-c_{5,v}>0$. Hence, we have $f_1(x)\neq0$ and $\|\frac{f_i(x)}{f_1(x)}\|_v<1$ for $2\leq i\leq N$.

For $1\leq i\leq N$, we get $\|g_i(x)\|_v \leq N(g_i)H_v(g_i)\|x_1\|_v^{d-k}$ because $\lambda_{H_\infty}(v,x)=\log\|x_1\|_v>0$. Therefore, we have $\log \|g_i(x)\|_v\leq (d-k)\lambda_{H_\infty}(v,x)+c_{6,v}$.

Now we can estimate for $2\leq i\leq N$ that 
\begin{align*}
-\log \|\frac{x_i'}{x_1'}-\frac{f_i(x)}{f_1(x)}\|_v 
&=-\log \|\frac{f_i(x)+g_i(x)}{f_1(x)+g_1(x)}-\frac{f_i(x)}{f_1(x)}\|_v \\
&= \log \|x_1'\|_v -\log \| g_i(x) - g_1(x)\frac{f_i(x)}{f_1(x)}\|_v\\
&\geq \lambda_{H_\infty}(v,f(x)) -(d-k)\lambda_{H_\infty}(v,x) - c_{6,v} -\log 2 \\
&\geq k\lambda_{H_\infty}(v,x)-c_{6,v}-c_{7,v}-\log 2.
\end{align*}
Since $\lambda_{f_{\infty}(P)}^1(v,f(x))\geq\lambda_P^1(v,x)-c_{5,v}\geq c_{3,v}-c_{5,v}>c_{1,v}$, we also have
$$
-\log\max\limits_{2\leq i\leq N}\|\frac{x_i'} {x_1'}\|_v=\lambda_{f_{\infty}(P)}^1(v,\pi(f(x)))\geq r\lambda_{H_{\infty}}(v,f(x))-c_{2,v}\geq r(d\lambda_{H_\infty}(v,x)-c_{7,v})-c_{2,v}.
$$
Therefore, we get
\begin{align*}
\lambda_{f_{\infty}(P)}^1(v,f_{\infty}(\pi(x)))
&= -\log\max\limits_{2\leq i\leq N}\|\frac{f_i(x)}{f_1(x)}\|_v \\
&\geq \min\{-\log\max\limits_{2\leq i\leq N}\|\frac{x_i'}{x_1'}-\frac{f_i(x)}{f_1(x)}\|_v,-\log\max\limits_{2\leq i\leq N}\|\frac{x_i'} {x_1'}\|_v\} -\log 2 \\
&\geq \min\{k\lambda_{H_\infty}(v,x)-c_{6,v}-c_{7,v}-\log 2, dr\lambda_{H_\infty}(v,x)-c_{2,v}-rc_{7,v} \}-\log2 \\
&\geq \min\{k,dr\}\lambda_{H_\infty}(v,x) - (c_{2,v}+c_{6,v}+(1+r)c_{7,v}+2\log 2).
\end{align*}

Since $\lambda_{P}^1(v,\pi(x))\geq \lambda_P^1(v,x)\geq c_{3,v}>c_{8,v}$, we have
$$
e_{f_\infty}(P)\lambda_P^1(v,\pi(x))+c_{9,v}\geq \lambda_{f_{\infty}(P)}^1(v,f_{\infty}(\pi(x))).
$$
Then we have
$$
\lambda_P^1(v,\pi(x)) +c_{4,v}\geq \frac{1}{e_{f_\infty}(P)} (\lambda_{f_{\infty}(P)}^1(v,f_{\infty}(\pi(x)))-c_{9,v}) +c_{4,v}\geq \frac{\min\{k,dr\}}{e_{f_\infty}(P)}\lambda_{H_\infty}(v,x)
$$
and finish the proof.
\end{proof}

We state the following result for future reference.

\begin{lemma}\label{3.16}
Let $P\in H_{\infty}(K)$. Then there exists an $S$-constant $(c_v)_{v\in S}$ such that $\lambda_{P}(v,x)\leq\lambda_{P}(v,\pi(x))+c_v$ holds for every $v\in S$ and every $x\in\mathbb{P}^N(K)\setminus\{O\}$.
\end{lemma}

\section{Proof of main results}

In this section, we prove Theorem \ref{main} and its consequences stated in Introduction.

\subsection{Proof of Theorem \ref{main}}

We prove Theorem \ref{main} in this subsection. We follow the 3 steps listed at the end of Introduction.

We use the notations in Setting \ref{setting}. In all the statements of this subsection, we implicitly assume that the regular endomorphism \emph{has degree} $d\geq2$ for simplicity. The case that $d=1$ is easy. It only appears in part (i) of Theorem \ref{main} and is considered in the proof of Theorem \ref{main}(i).

\begin{lemma}\label{4.1}
Let $K$ be an arithmetic function field. Let $f$ be a regular endomorphism of $\mathbb{A}^N$ defined over $K$, and let $x\in\mathbb{A}^N(K)$ be a point. Let $C\subseteq\mathbb{A}_K^{N}$ be an irreducible closed subcurve such that $C$ has an infinite intersection with the orbit $\mathcal{O}_f(x)$. Then the degree of the zero cycle $[(\overline{C}\cap H_{\infty})_{\mathrm{red}}]$ is $1$ or $2$. Furthermore, all points in $\overline{C}\cap H_{\infty}$ are $f_{\infty}$-periodic.
\end{lemma}

\begin{proof}
The degree of $[(\overline{C}\cap H_{\infty})_{\mathrm{red}}]$ is just the number $|(\overline{C}\cap H_{\infty})(\overline{K})|$. Theorem \ref{siegel} guarantees that it is $1$ or $2$. In order to prove that the points in $\overline{C}\cap H_{\infty}$ are $f_{\infty}$-periodic, we consider a sequence of irreducible closed curves in $\mathbb{A}_K^N$
$$
\cdots\stackrel{f}\rightarrow C_{-n}\stackrel{f}\rightarrow\cdots\stackrel{f}\rightarrow C_{-1}\stackrel{f}\rightarrow C_0=C
$$
defined as follows.

Firstly, we can find an irreducible component $C_{-1}$ of $f^{-1}(C)$ which has an infinite intersection with the orbit $\mathcal{O}_f(x)$. Since $f$ is finite, we know $C_{-1}$ is a curve. Repeating this procedure, we can find a sequence of curves as above. From the construction, we see that each $C_{-n}$ contains infinitely many points in $\mathcal{O}_f(x)$.

We regard $f$ as an endomorphism of $\mathbb{P}_K^N$. Then we get finite maps $f|_{\overline{C_{-n-1}}}:\overline{C_{-n-1}}\rightarrow\overline{C_{-n}}$. As these maps are surjective, we see that $f_{\infty}$ surjectively maps $\overline{C_{-n-1}}\cap H_{\infty}$ onto $\overline{C_{-n}}\cap H_{\infty}$. Let $P_0$ be a point in $\overline{C}\cap H_{\infty}$. Then we can get a sequence of points $P_{-n}\in\overline{C_{-n}}\cap H_{\infty}$ which satisfies $f_{\infty}(P_{-n-1})=P_{-n}$ for every $n$. The first part of this lemma guarantees that $[\kappa(P_{-n}):K]\leq2$ for every $n$, where $\kappa(P_{-n})$ denotes the residue field of $P_{-n}\in H_{\infty}$. As we have assumed that the algebraic degree $d$ of $f$ is at least $2$, the map $f_{\infty}$ is a polarized endomorphism of $H_{\infty}$. Therefore, a standard argument of Moriwaki's height theory \cite{Mor00} deduces that $\{P_{-n}|\ n\geq0\}$ is a finite set, in which the key point is Northcott's finiteness property \ref{northcott}. Hence $P_0$ is $f_{\infty}$-periodic. This finishes the proof.
\end{proof}

\proof[Proof of Theorem \ref{main}(i)]
We find an arithmetic function field $K\subseteq\mathbb{C}$ such that all data are defined over $K$, and $I=\overline{C}\cap H_{\infty}$ contains only $K$-rational points. Then Lemma \ref{4.1} says that $1\leq|I|\leq2$.

If the degree $d=1$, i.e. $f$ is an affine automorphism, then it is well-known that the DML conjecture holds for $f$. For example, see \cite{BGT10}. Hence the curve $\overline{C}$ is $f$-periodic. This in particular implies that the closed points in $I$ are $f_{\infty}$-periodic.

If the degree $d\geq2$, then we use Lemma \ref{4.1} to conclude that the points in $I$ are $f_{\infty}$-periodic.
\endproof

Now we turn to the proof of Theorem \ref{main}(ii). We construct the following setting. Same as in the proof above, we let $K\subseteq\mathbb{C}$ be an arithmetic function field such that all data are defined over $K$ and $I=\overline{C}\cap H_{\infty}$ contains only $K$-rational points. As $1\leq|I|\leq2$, we may write $I=\{P_0,Q_0\}$ for some points $P_0,Q_0\in H_{\infty}(K)$. If $|I|=1$, then we read $P_0=Q_0$ as a matter of convention. Since $P_0$ and $Q_0$ are $f_{\infty}$-periodic, the set $I_0=\mathcal{O}_{f_{\infty}}(P_0)\cup\mathcal{O}_{f_{\infty}}(Q_0)$ is finite and contains only periodic points. For $n\geq0$, we let $P_{-n}$ (resp. $Q_{-n}$) be the unique element in $I_0$ such that $f_{\infty}^n(P_{-n})=P_0$ (resp. $f_{\infty}^n(Q_{-n})=Q_0$). Then for every $n$, we have $f_{\infty}(P_{-n-1})=P_{-n}$ and $f_{\infty}(Q_{-n-1})=Q_{-n}$. Indeed, the sequences $(P_{-n})_{n\geq0}$ and $(Q_{-n})_{n\geq0}$ are the backward orbits of $P_0$ and $Q_0$ that run along their periods, respectively.

As in the proof of Lemma \ref{4.1}, we consider the sequence of irreducible closed curves in $\mathbb{A}_K^N$
$$
\cdots\stackrel{f}\rightarrow C_{-n}\stackrel{f}\rightarrow\cdots\stackrel{f}\rightarrow C_{-1}\stackrel{f}\rightarrow C_0=C.
$$
The way of construction there guarantees that each $C_{-n}$ has an infinite intersection with the orbit $\mathcal{O}_f(x)$. Same as in the proof of Lemma \ref{4.1}, we have $f_{\infty}(\overline{C_{-n-1}}\cap H_{\infty})=\overline{C_{-n}}\cap H_{\infty}$ for every $n$. Also, Lemma \ref{4.1} says that each set $\overline{C_{-n}}\cap H_{\infty}$ consists of $1$ or $2$ $f_{\infty}$-periodic points.

\begin{lemma}\label{4.2}
We have $\overline{C_{-n}}\cap H_{\infty}=\{P_{-n},Q_{-n}\}$ for every $n$. In particular, each set $\overline{C_{-n}}\cap H_{\infty}$ contains only $K$-rational points.
\end{lemma}

\begin{proof}
We prove by induction on $n$. The case that $n=0$ is valid. Suppose we know $\overline{C_{-n}}\cap H_{\infty}=\{P_{-n},Q_{-n}\}$. As the points in $\overline{C_{-n-1}}\cap H_{\infty}$ are $f_{\infty}$-periodic and are sent to either $P_{-n}$ or $Q_{-n}$ by $f_{\infty}$, we see that $\overline{C_{-n-1}}\cap H_{\infty}\subseteq\{P_{-n-1},Q_{-n-1}\}$. The inclusion is indeed an equality because $f_{\infty}$ surjectively maps $\overline{C_{-n-1}}\cap H_{\infty}$ onto $\overline{C_{-n}}\cap H_{\infty}$. Thus we finish the proof by induction.
\end{proof}

Next, we show that the curves $C_{-n}$ become more and more vertical.

Consider the following two sequences of positive real numbers. Let $r_{1,0}$ (resp. $r_{2,0}$) be the reciprocal of the intersection number $i(P_0)$ (resp. $i(Q_0)$) of $C=C_0$ and $H_{\infty}$ at $P_0$ (resp. $Q_0$). For $n\geq1$, we let $r_{1,n}=\frac{\min\{k,dr_{1,n-1}\}}{e_{f_{\infty}}(P_{-n})}$ and $r_{2,n}=\frac{\min\{k,dr_{2,n-1}\}}{e_{f_{\infty}}(Q_{-n})}$. Hence we get $(r_{1,n})_{n\geq0},(r_{2,n})_{n\geq0}\in\mathbb{R}_{+}^{\mathbb{N}}$.

Recall $I_0=\mathcal{O}_{f_{\infty}}(P_0)\cup\mathcal{O}_{f_{\infty}}(Q_0)$ is a finite set. We denote $r=\min\limits_{P\in I_0}\frac{k}{e_{f_{\infty}}(P)}$.

\begin{lemma}\label{4.3}
We have $r>2$. Moreover, we have $r_{1,n}\geq r$ and $r_{2,n}\geq r$ for sufficiently large $n$.
\end{lemma}

\begin{proof}
The assertion $r>2$ follows from our assumption. Notice that base extension does not change the multiplicities at rational points.

We have $d\geq k>2\max\limits_{P\in I_0}e_{f_{\infty}}(P)$. For every $n\geq1$, we have either $r_{1,n}\geq r$ or $r_{1,n}\geq rr_{1,n-1}$. Since $r_{1,0}>0$, we see that $r_{1,n}\geq r$ holds for sufficiently large $n$. The proof for $r_{2,n}$ is the same.
\end{proof}

We want to use the machinery of local heights established in Section 3. We fix a finite set $S_0\subseteq M_{K}^0$ such that the orbit $\mathcal{O}_f(x)$ contains only $S$-integral points in $\mathbb{A}^N(K)$, in which $S=M_{K}^{\infty}\sqcup S_0$. We use the notions and results in Section 3 towards this set $S$ of places.

\begin{lemma}\label{4.4}
There exist two sequences $((c_{P_{-n},1,v})_{v\in S})_{n\geq0},((c_{P_{-n},2,v})_{v\in S})_{n\geq0}$ of $S$-constants that satisfy the following property. For every $v\in S$ and every $x\in C_{-n}(K)$, we have $x\neq O$ and
$$
r_{1,n}\lambda_{H_{\infty}}(v,x)\leq\lambda_{P_{-n}}(v,\pi(x))+c_{P_{-n},2,v}
$$
when $\lambda_{P_{-n}}(v,x)\geq c_{P_{-n},1,v}$.

There also exist two sequences $((c_{Q_{-n},1,v})_{v\in S})_{n\geq0},((c_{Q_{-n},2,v})_{v\in S})_{n\geq0}$ of $S$-constants which satisfy the same statement after changing every $P_{-n}$ into $Q_{-n}$ and changing $r_{1,n}$ into $r_{2,n}$.
\end{lemma}

\begin{proof}
We prove by induction on $n$. The case $n=0$ can be verified by combining Lemma \ref{3.10}, Lemma \ref{3.11} and Lemma \ref{3.16}. According to the definitions of $r_{1,n}$ and $r_{2,n}$, we see that the induction step is completed by using Proposition \ref{3.15} and letting the set $X$ in there be the curves $C_{-n}$. Thus we finish the proof by induction.
\end{proof}

The following proposition shows that once a curve looks sufficiently vertical at its infinity points, then it must be a vertical line. The proof uses Roth's theorem. Matsuzawa summarized a version of Roth's theorem in terms of Weil functions \cite[Theorem A.1]{Mat25}. Please see it to get a general feeling about this kind of argument. We need to be a little bit more careful due to the setting of arithmetic function fields.

\begin{proposition}\label{4.5}
Let $C\subseteq\mathbb{A}_K^N$ be an irreducible closed curve which contains infinitely many $S$-integral rational points. Let $\overline{C}$ be the closure of $C$ in $\mathbb{P}_K^N$ and assume that $\overline{C}\cap H_{\infty}=\{P_1,\dots,P_t\}$ contains only rational points. Let $r>2$ be a real number and let $(c_{1,v})_{v\in S}$ and $(c_{2,v})_{v\in S}$ be two $S$-constants. Suppose that for every $1\leq j\leq t$, every $v\in S$, and every $x\in C(K)$, we have $x\neq O$ and
$$
r\lambda_{H_{\infty}}(v,x)\leq\lambda_{P_j}(v,\pi(x))+c_{2,v}
$$
when $\lambda_{P_j}(v,x)\geq c_{1,v}$. Then $C$ is a vertical line.
\end{proposition}

\begin{proof}
We can move the rational points $P_1,\dots,P_t$ into the affine chart $\{x_1\neq0\}$ of $\mathbb{P}_K^N$ by applying a $\mathbb{Z}$-coefficient linear transformation of $\mathbb{A}_K^N$ (we do not need its inverse to also has $\mathbb{Z}$-coefficients). Since the transformation has integer coefficients, the curve still has infinitely many $S$-integral points after the modification. Lemma \ref{3.7} and Lemma \ref{3.9} guarantee that the assumption is still valid. Also, the conclusion does not change. Hence we may assume $P_1,\dots,P_t\in\{x_1\neq0\}$ without loss of generality. For $1\leq j\leq t$, we write $P_j=[0,1,a_{j,2},\dots,a_{j,N}]$. Further applying a $\mathbb{Z}$-coefficient linear transformation, we may assume that those $t$ coefficients $a_{1,i},\dots,a_{t,i}$ are pairwise distinct for every $2\leq i\leq N$.

\begin{claim}\label{4.6}
There exists an $S$-constant $(c_v)_{v\in S}$ such that $r\lambda_{H_{\infty}}(v,x)\leq\sum\limits_{j=1}^t\lambda_{P_j}(v,\pi(x))+c_{v}$ holds for every $v\in S$ and every $x\in C(K)\setminus\{O\}$.
\end{claim}

\begin{proof}
Proposition \ref{3.13} says that there is an $S$-constant $(c_v')_{v\in S}$ such that
$$
\lambda_{H_{\infty}}(v,x)\leq\sum\limits_{j=1}^{t}i(P_j)\lambda_{P_j}(v,x)+c_v'
$$
holds for every $v\in S$ and every $x\in C(K)$. We claim that letting $c_v=rc_{1,v}\sum\limits_{j=1}^{t}i(P_j)+rc_v'+c_{2,v}$ suffices. Indeed, under such situation, we may assume $\lambda_{H_{\infty}}(v,x)\geq c_{1,v}\sum\limits_{j=1}^{t}i(P_j)+c_v'$ as otherwise the inequality trivially holds. Then there exists $j\in\{1,\dots,t\}$ such that $\lambda_{P_j}(v,x)\geq c_{1,v}$. Then the expected inequality is an immediate consequence of the assumption of Proposition \ref{4.5}.
\end{proof}

According to Claim \ref{4.6}, we may fix an $S$-constant $(c_v)_{v\in S}$ such that
$$
r\lambda_{H_{\infty}}(v,x)\leq\sum\limits_{j=1}^t\lambda_{P_j}^1(v,\pi(x))+c_{v}
$$
holds for every $v\in S$ and every $x\in C(K)\setminus\{O\}$. Let $x=[1,x_1,\dots,x_N]$ be an $S$-integral point in $C(K)\setminus\{O\}$. Denote $c=\int_{S}c_vd\mu(v)$, which is a nonnegative real number.
\begin{enumerate}
\item
If $x_1=0$, then we have $r\log^+\max\limits_{1\leq i\leq N}\|x_i\|_v\leq c_v$. Since $x$ is an $S$-integral point, we can integrate and get $rh_{K}(x)\leq c$ where $h_K(x)$ is the naive height function introduced in Definition \ref{2.1}. So such rational points $x$ have a bounded height. According to Northcott's finiteness property \ref{northcott}, there are only finitely many such $S$-integral points with $x_1=0$.

\item
If $x_1\neq0$, we denote $y_i=\frac{x_i}{x_1}$ for every $2\leq i\leq N$. We have $r\log^+\max\limits_{1\leq i\leq N}\|x_i\|_v\leq-\sum\limits_{j=1}^t\log^-\max\limits_{2\leq i\leq N}\|y_i-a_{j,i}\|_v+c_v$. In particular, for every $2\leq i\leq N$ we have
$$
r\log^+\max\limits_{1\leq i\leq N}\|x_i\|_v\leq-\sum\limits_{j=1}^t\log^-\|y_i-a_{j,i}\|_v+c_v.
$$

Taking integration as above, we get
$$
rh_K(x)\leq\int_{S}\left(\sum\limits_{j=1}^{t}-\log^-\|y_i-a_{j,i}\|_v\right)d\mu(v)+c.
$$
We have $h_K(x)=h_K([1,x_1,\dots,x_N])\geq h_K([x_1,\dots,x_N])=h_K([1,y_2,\dots,y_N])\geq h_K(y_i)$ for every $2\leq i\leq N$. Since $r>2$ and we have assumed $a_{1,i}\dots,a_{t,i}$ to be pairwise distinct, Roth's Theorem \ref{roth} guarantees that there are only finitely many possible values for each $y_i$. Hence in this case, the projection $\pi(x)$ of those $S$-integral points lie in a finite set.
\end{enumerate}

According to the discussion above, those infinitely many $S$-integral points in $C(K)\setminus\{O\}$ project to finitely many points in $H_{\infty}(K)$ under $\pi$. Assume there are infinitely many such points which project to $P_0\in H_{\infty}(K)$. Then $C\subseteq\mathbb{A}_K^N$ must be the vertical line passing through $P_0$. Thus we finish the proof.
\end{proof}

\begin{remark}
In fact, we have $1\leq t\leq 2$ under the assumption of Proposition \ref{4.5}. This can be seen by looking into the proof of Siegel's theorem \cite[Corollary 4.11]{Voj21}. We do not need this fact in the proof above.
\end{remark}

Now we can prove Theorem \ref{main}(ii).

\proof[Proof of Theorem \ref{main}(ii)]
We continue the previous discussion. Recall we have a sequence of irreducible closed curves in $\mathbb{A}_K^N$
$$
\cdots\stackrel{f}\rightarrow C_{-n}\stackrel{f}\rightarrow\cdots\stackrel{f}\rightarrow C_{-1}\stackrel{f}\rightarrow C_0=C,
$$
each of which has an infinite intersection with the orbit $\mathcal{O}_f(x)$. In particular, each of these curves contains infinitely many $S$-integral points.

Lemma \ref{4.2} says that $\overline{C_{-n}}\cap H_{\infty}=\{P_{-n},Q_{-n}\}$ contains only rational points. In the setting of Lemma \ref{4.4}, we denote $c_{n,1,v}=\max\{c_{P_{-n},1,v},c_{Q_{-n},1,v}\}$ and $c_{n,2,v}=\max\{c_{P_{-n},2,v},c_{Q_{-n},2,v}\}$. According to Lemma \ref{4.3} and Lemma \ref{4.4}, these two sequences of $S$-constants satisfy the following property: for every sufficiently large $n$, every $R\in\{P_{-n},Q_{-n}\}$, every $v\in S$, and every $x\in C_{-n}(K)$, we have $x\neq O$ and
$$
r\lambda_{H_{\infty}}(v,x)\leq\lambda_{R}(v,\pi(x))+c_{n,2,v}
$$
when $\lambda_{R}(v,x)\geq c_{n,1,v}$. Then we can use Proposition \ref{4.5} to deduce that every $C_{-n}$ is a vertical line for sufficiently large $n$. In particular, we must have $P_{-n}=Q_{-n}$ for sufficiently large, and hence every $n$.

By the definition of the sequence $(P_{-n})_{n\geq0}$, it is the backward orbit that runs along the period of $P_0$ under the action of $f_{\infty}$. We know $C_{-n}$ is the vertical line passing through $P_{-n}$ for sufficiently large $n$. Therefore, the curves $C_{-n}$ form a cycle under the action of $f$. In particular, every $C_{-n}$ is the vertical line passing through $P_{-n}$ and they are all $f$-periodic curves. Taking $n=0$, we see that $C$ is an $f$-periodic vertical line. Returning to the setting over $\mathbb{C}$, the same assertions remain valid and we finish the proof.
\endproof

\subsection{Proofs of Corollary \ref{cor main} and Theorem \ref{main2}}

\proof[Proof of Corollary \ref{cor main}]
If $d\geq3$, then this is a direct consequence of Theorem \ref{main}. Indeed, we have $k=3$ in this setting. Moreover, not only $I=\overline{C}\cap H_{\infty}$ but also $I_0=\bigcup\limits_{P\in I}\mathcal{O}_{f_{\infty}}(P)$ is contained in the open set $\{x_1\cdots x_N\}\neq0$. Thus $k=3>2=2e_{f_{\infty}}(P)$ holds for every $P\in I_0$ and then we can apply Theorem \ref{main} to get the result. If $d=2$, then $g_1(x_1,\dots,x_N)=\cdots=g_N(x_1,\dots,x_N)=0$ and we may consider the iteration $f^2$ and reduce to the case in which $d\geq3$.

The assertion about periodic curves is then an immediate consequence as there are non-preperiodic points in every curve.
\endproof

\proof[Proof of Theorem \ref{main2}]
This is also a direct consequence of Theorem \ref{main}. Notice Theorem \ref{main}(i) guarantees that we only need to control the multiplicities at the periodic points of $f_{\infty}$.
\endproof

\section{Further discussions: the positive characteristic case}

It is also interesting to study the dynamical Mordell--Lang conjecture over fields of positive characteristic. In this setting, we still concern about the return set of an orbit into a closed subvariety, but the conclusion needs to be modified. See, for example, \cite[Example 3.4.5.1]{BGT16}. Indeed, the examples in \cite[Section 5]{XY} show that the return set can be \emph{very} complicated, instead of a finite union of arithmetic progressions. Although such complicated return sets often appear for endomorphisms of zero entropy, there also exist polarized counterexamples in positive characteristic. See \cite[Example 5.4]{XY}. The interested reader may have a glance at \cite{CGSZ21,XY25,XY,XYb} and the references therein for recent progress on the dynamical Mordell--Lang problem in positive characteristic.

As we have mentioned in Introduction, the authors studied the dynamical Mordell--Lang conjecture in \cite{XY} by using a height argument. The results in loc. cit. are characteristic-free. Therefore, we would like to know about how much we can extend the results in this work to the setting of positive characteristic. Nevertheless, there is neither Siegel's theorem nor Roth's theorem in positive characteristic. Hence our approach should fail.

Our main result Theorem \ref{main2} can be divided into two parts. Firstly, we prove the DML(1) property under Condition \ref{conditon k}; and secondly, we characterize the periodic curves in this setting --- all of them are vertical lines. The following example shows that we cannot expect the second part in positive characteristic. Indeed, it shows that in the setting of Theorem \ref{main}, there are invariant curves which are not straight lines while the degree gap is arbitrarily large.

\begin{example}
Let $K$ be an algebraically closed field of characteristic $p>0$. Let $d$ be a positive integer prime to $p$ and let $q>d$ be a power of $p$. Consider the endomorphism
$$
f(x,y)=(x^q+(x^d-y^d-1)y^{q-d},y^q)
$$
of $\mathbb{A}^2$. It is a regular endomorphism with degree gap $d$.

Let $C$ be the irreducible affine curve $\{x^d=y^d+1\}$. Then $C$ is an $f$-invariant curve which has degree $d$. We can see that $\overline{C}\cap H_{\infty}\subseteq\{xy\neq0\}$. Hence $f_{\infty}$ is unramified at the points in $\overline{C}\cap H_{\infty}$. This shows that neither of the two parts of Theorem \ref{main} is valid in positive characteristic.
\end{example}

In the setting of Theorem \ref{main}(i), the points at infinity of $C$ \emph{have to} be $f_{\infty}$-periodic because the DML conjecture asserts that $C$ is $f$-periodic. In positive characteristic, the DML conjecture in this form fails, but we once wonder whether the statement in the first half of the last sentence may still hold. The previous example does not rule out this possibility, but please see the following example.

\begin{example}\label{AIexample}
Let $p$ be a prime and let the base field $K=\overline{\mathbb{F}_p(t)}$. Consider the endomorphism
$$
f(x,y)=(x^p,y^p+x^{p-1}y)
$$
of $\mathbb{A}^2$. Let $C$ be the line $y=1+tx$ and let the starting point $\alpha=(t^{-1},1)$. One can check that $f^{p^n}(\alpha)\in C$ for every $n\geq0$. Hence $\mathcal{O}_f(\alpha)\cap C$ is infinite, while the intersection point $[0,1,t]$ of $\overline{C}$ and $H_{\infty}$ is not $f_{\infty}$-periodic,
\end{example}

These examples illustrate that not only Roth's theorem, but also Siegel's theorem play an important role in our approach.

However, a variation of the DML part of Theorem \ref{main2} may remain valid in positive characteristic. We want to propose the following question.

\begin{question}
Let $f$ be a regular endomorphism of $\mathbb{A}^N$ over an algebraically closed field $K$ of arbitrary characteristic. We use the notations in Setting \ref{setting}. Assume that the algebraic degree $d>\max\limits_{x\in H_{\infty}(K)}e_{f_{\infty}}(x)$. Does $f$ satisfy the DML(1) property?
\end{question}

\begin{remark}
\begin{enumerate}
\item
In Condition \ref{conditon k}, we only bound the multiplicities at periodic points of $f_{\infty}$. But as Example \ref{AIexample} suggests, the condition here should bound the multiplicities for all points on the infinity hyperplane.

\item
In Condition \ref{conditon k} we bound the multiplicities by the degree gap $k$ while here we bound them by the algebraic degree $d$, which is more natural and essential. The endomorphism $f$ in Example \ref{AIexample} is homogeneous, which implies that the degree gap may not play an important role here.

\item
To our knowledge, this question remains open in characteristic $0$. In this case, we may only bound the multiplicities at periodic points of $f_{\infty}$. Recently, our A.I. claims a proof of this statement. We are still learning about it.
\end{enumerate}
\end{remark}

New methods are needed for giving a positive answer of this question in positive characteristic.

\section{Appendix: proofs of the lemmas in Section 3}

In this section, we provide proofs of the lemmas in Section 3.

\proof[Proof of Lemma \ref{3.3}]
Without loss of generality, we may assume $i=0$, $j=1$, and $a_0=1$. It suffices to prove that there is an $S$-constant $(c_v)_{v\in S}$ such that $\lambda_P^0(v,x)\leq \lambda_P^1(v,x)+c_v$ for every $v\in S$ and $x\in\mathbb{P}^N(K)$. We claim that taking
$$
c_v=\sum\limits_{i=1}^{N}\log^+\|a_i\|_v+2\log^+\|\frac{1}{a_1}\|_v+\log4
$$
suffices.

Let $x=[x_0,\dots,x_N]\in\mathbb{P}^N(K)$. In order to prove the inequality, we may assume $\lambda_P^0(v,x)\geq\log^+\|1/a_1\|_v+\log2$. As $\lambda_P^0(v,x)>0$, we may assume $x_0=1$ and then we have $\lambda_P^0(v,x)=-\log \max\limits_{1\leq i\leq N}\|a_i-x_i\|_v$. We get $\|a_1-x_1\|_v\leq \frac{\|a_1\|_v}{2}$. Hence $x_1\neq0$ and
$$
\lambda_P^1(v,x) = -\log^- \max\{ \|\frac{1}{a_1}-\frac{1}{x_1}\|_v, \|\frac{a_2}{a_1}-\frac{x_2}{x_1}\|_v,\cdots, \|\frac{a_N}{a_1}-\frac{x_N}{x_1}\|_v\}.
$$
    
We have $\|\frac{1}{a_1}-\frac{1}{x_1}\|_v = \|\frac{a_1-x_1}{a_1x_1}\|_v \leq \frac{2\|a_1-x_1\|_v}{\|a_1\|_v^2}$ and
\begin{align*}
\max_{2\leq i \leq N} \|\frac{a_i}{a_1}-\frac{x_i}{x_1}\|_v
&= \max_{2\leq i \leq N} \|\frac{a_ix_1-a_1x_i}{a_1x_1}\|_v
= \max_{2\leq i \leq N} \|\frac{(a_i-x_i)a_1-(a_1-x_1)a_i}{a_1x_1}\|_v \\
&\leq \frac{2\max\limits_{1\leq i \leq N} \|a_i-x_i\|_v \max\limits_{1\leq i \leq N}\|a_i\|_v}{\|a_1x_1\|_v} 
\leq \frac{4\max\limits_{1\leq i\leq N}\|a_i-x_i\|_v\max\limits_{1\leq i \leq N} \|a_i\|_v}{\|a_1\|_v^2}.
\end{align*}

Thus $\max \{ \|\frac{1}{a_1}-\frac{1}{x_1}\|_v, \max\limits_{2\leq i \leq N} \|\frac{a_i}{a_1}-\frac{x_i}{x_1}\|_v\} \leq \frac{4\max\limits_{1\leq i\leq N}\|a_i-x_i\|_v\max\{1,\max\limits_{1\leq i \leq N} \|a_i\|_v\}}{\|a_1\|_v^2}$ and then we have 
$$
\lambda_P^1(v,x) \geq \lambda_P^0(v,x) - \log^+ \max_{1\leq i\leq N}\|a_i\|_v - \log \frac{4}{\|a_1\|_v^2}.
$$
This finishes the proof.
\endproof

\proof[Proof of Lemma \ref{3.7}]
It suffices to prove $\lambda_P(v,x)\leq \lambda_{A(P)}(v,A(x))+c_v$ for some $S$-constant $(c_v)_{v\in S}$. Write $P=[a_0,\dots, a_N]$, $x=[x_0,\dots,x_N]$, and $A=(t_{ij})_{0\leq i,j\leq N}$. For $0\leq i\leq N$, we denote $a_i'=\sum\limits_{j=0}^{N}t_{ij}a_j$ and $x_i'=\sum\limits_{j=0}^{N}t_{ij}x_j$. Then $A(P)=[a_0',\dots,a_N']$ and $A(x)=[x_0',\dots,x_N']$. Assume $a_ia_j'\neq0$. We only deal with the case that $i=j=0$ as the proof of other cases are the same. Further assume $a_0=1$.
    
We prove that $\lambda_P^0(v,x)\leq \lambda_{A(P)}^0(v,A(x))+c_v$ in which
$$
c_v=\sum\limits_{0\leq j,k\leq N}\log^+\|t_{jk}\|_v+\sum\limits_{j=0}^{N}\log^+\|a_j'\|_v+2\log^+\|\frac{1}{a_0'}\|_v+\log(4N+4).
$$

We denote $C_v=\frac{\|a_0'\|_v}{2\sum\limits_{j=0}^N \|t_{0j}\|_v}$. We may assume that $\lambda_P^0(v,x)\geq c_v\geq -\log^-C_v+\log2>0$. Then we have $x_0\neq0$, and we may assume $x_0=1$. Then $\lambda_P^0(v,x)=-\log\max\limits_{1\leq j\leq N}\|a_j-x_j\|_v$ and $\|a_j-x_j\|_v \leq C_v$ for every $1\leq j\leq N$.
    
Now $\|a_0'-x_0'\|_v = \|\sum\limits_{j=0}^N t_{0j}(a_j-x_j)\|_v \leq \sum\limits_{j=0}^N \|t_{0j}\|_v C_v= \frac{1}{2}\|a_0'\|_v$. For $1\leq i\leq N$, we have
\begin{align*}
\|\frac{a_i'}{a_0'}-\frac{x_i'}{x_0'}\|_v
&= \|\frac{a_i'(x_0'-a_0')-a_0'(x_i'-a_i')}{a_0'x_0'}\|_v
= \|\frac{a_i'\sum\limits_{j=1}^N t_{0j}(x_j-a_j) - a_0'\sum\limits_{j=1}^N t_{ij}(x_j-a_j)}{a_0'x_0'}\|_v \\
&\leq \frac{2N\max\limits_{0\leq j,k\leq N}\|t_{jk}\|_v\max\limits_{0\leq j\leq N} \|a_j'\|_v \max\limits_{1\leq j\leq N}\|x_j-a_j\|_v}{\|a_0'\|_v^2/2}.
\end{align*}

Taking logarithm, we get
$$
\lambda_{A(P)}^0(v,A(x)) + \log\frac{4N\max\limits_{0\leq j,k\leq N}\|t_{jk}\|_v\max\limits_{0\leq j\leq N} \|a_j'\|_v}{\|a_0'\|_v^2}\geq \lambda_P^0(v,x)
$$
and thus finish the proof.
\endproof

In the following proofs, we use Notation \ref{3.12} for simplicity.

\begin{proof}[Proof of Lemma \ref{3.8}]
According to Lemma \ref{3.7}, we can composite $f$ with certain automorphisms and assume $P=f(P)=[1,0,\dots,0]$. Write $x=[x_0,\dots,x_N]$ and
$$
f(x)=[x_0^d+f_0(x_0,\dots,x_N), f_1(x_0,\dots,x_N),\dots,f_N(x_0,\dots,x_N)]=[x_0',\dots,x_N'],
$$
where $f_i(1,0,\dots,0)=0$ for $0\leq i\leq N$.

\begin{enumerate}
\item
We prove that $\lambda_P^0(v,x)\leq \lambda_{f(P)}^0(v,f(x))+c_v$, where
$$
c_v=\sum\limits_{i=0}^{N}\left(l_v(f_i)+\log^+N(f_i)\right)+\log2.
$$
If $x_0=0$, then $\lambda_P^0(v,x)=0$ and the assertion follows. So we can assume $x_0\neq0$, or equivalently $x_0=1$.

Denote $C_v=\min \{1,\frac{1}{2N(f_0)H_v(f_0)}\}$. We assume that $\lambda_P^0(v,x)\geq c_v\geq -\log^- C_v$. Then $\|x_i\|_v \leq C_v\leq1$ for every $1\leq i\leq N$. Since $\|f_i(x_0,\dots,x_N)\|_v\leq N(f_i)H_v(f_i)\max\limits_{1\leq i\leq N}\|x_i\|_v$, we have $\|x_0'\|_v = \|1+f_0(x_0,\dots,x_N)\|_v\geq 1-N(f_0)H_v(f_0)\max\limits_{1\leq i\leq N}\|x_i\|_v\geq \frac{1}{2}$. Thus,  
$$
\|\frac{x_i'}{x_0'}\|_v=\|\frac{f_i(x_0,\dots,x_N)}{x_0'}\|_v\leq 2N(f_i)H_v(f_i)\max\limits_{1\leq i\leq N} \|x_i\|_v
$$
for $1\leq i\leq N$.

Taking logarithm, we get
$$
\lambda_{f(P)}^0(v,f(x)) + \log \max\limits_{1\leq i\leq N} 2N(f_i)H_v(f_i)\geq \lambda_P^0(v,x)
$$
and hence finish the proof.
    
\item
For $1\leq i \leq N$, set $g_i(x_1,\dots, x_N)= f_i(1,x_1,\dots,x_N)$. Write $e=e_f(P)$ for simplicity. Then for $1\leq i,j\leq N$, there are $r_{ij}\in  K[x_1,\dots,x_N]$ and $s_{ij}\in (x_1,\dots,x_N)K[x_1,\dots,x_N]$ such that 
$$
x_i^{e} = \sum_{j=1}^N \frac{g_j(x_1,\dots,x_N)}{1+g_0(x_1,\dots,x_N)}\cdot\frac{r_{ij}(x_1,\dots,x_N)}{1+s_{ij}(x_1,\dots,x_N)}
$$
holds for every $1\leq i\leq N$.

Let $c_{1,v}=\sum\limits_{1\leq i,j\leq N}\left(l_v(s_{ij})+\log^+N(s_{ij})\right)+\log2$ and $c_{2,v}=\sum\limits_{1\leq i,j\leq N}\left(l_v(r_{ij})+\log^+N(r_{ij})\right)+\log(2N)$. We prove that when $\lambda_P^0(v,x)\geq c_{1,v}$, we have $\lambda_{f(P)}^0 (v,f(x))\leq e\lambda_P^0(v,x)+c_{2,v}$.

We may assume $x_0=1$. By the assumption that $\lambda_P^0(v,x)\geq c_{1,v}$, for each $1\leq i\leq N$, we have $\|x_i\|_v\leq \min\limits_{1\leq i,j\leq N} \frac{1}{2N(s_{ij})H_v(s_{ij})}$ and $\|x_i\|_v\leq 1$. So for every $1\leq i,j\leq N$, we have $\|r_{ij}(x_1,\dots, x_N)\|_v\leq N(r_{ij})H_v(r_{ij})$ and
$$
\|1+s_{ij}(x_1,\dots,x_N)\|_v\geq 1-N(s_{ij})H_v(s_{ij})\max\limits_{1\leq j\leq N}\|x_i\|_v \geq\frac{1}{2}.
$$

To prove the assertion, we may assume $x_0'=1$. We have $x_i'=\frac{g_i(x_1,\dots,x_N)}{1+g_0(x_1,\dots,x_N)}$.
Hence
$$
\|x_i\|_v^e\leq \sum\limits_{j=1}^N \|x_j'\|_v \|\frac{r_{ij}}{1+s_{ij}}\|_v \leq N \max\limits_{1\leq j\leq N}\|x_j'\|_v \max\limits_{1\leq i,j\leq N}2N(r_{ij})H_v(r_{ij})
$$
holds for every $1\leq i\leq N$. Taking logarithm, we get
$$
\lambda_{f(P)}^0(v,f(x)) \leq e\lambda_P^0(v,x) +\log^+ \left( 2N\max\limits_{1\leq i,j \leq N}N(r_{ij})H_v(r_{ij}) \right) 
$$
and finish the proof.
\end{enumerate}
\end{proof}

Lemma \ref{3.9} is about the Weil functions for divisors. We omit its proof because it is a direct consequence of \cite[Theorem 3.24]{Voj21}. The interested reader can also prove it by a standard calculation using Hilbert's Nullstellensatz.

\proof[Proof of Lemma \ref{3.10}]
According to Lemma \ref{3.7}, we may apply an automorphism on $\mathbb{P}^N$ and assume $P=[1,0,\dots,0]$. Then there exists a homogeneous polynomial $g$ in the ideal of $V$ such that $g(P)\neq 0$. We write $g(x_0,\dots,x_N)=x_0^d+g_0(x_0,\dots,x_N)$ in which $g_0(1,0,\dots,0)=0$. We prove $g(x)\neq 0$ when $\lambda_P^0(v,x)\geq c_v$, where $c_v=l_v(g_0)+\log^+N(g_0)+\log2$.

Write $x=[x_0,\dots,x_N]$. By the assumption that $\lambda_P^0(v,x)\geq c_v > 0$, we have $x_0\neq0$. So we assume $x_0=1$ and get $\|x_i\|_v\leq \frac{1}{2N(g_0)H_v(g_0)}$ for each $1\leq i\leq N$. Hence 
$$
\|g(x)\|_v \geq 1 - N(g_0)H_v(g_0) \max\limits_{1\leq i\leq N} \|x_i\|_v \geq \frac{1}{2}
$$
and we finish the proof.
\endproof

\proof[Proof of Lemma \ref{3.11}]
Applying a linear transformation, we may assume $P=[0,1,0,\dots,0]$ by taking Lemma \ref{3.7} and Lemma \ref{3.9} into account. Write $x=[1,x_1,\dots,x_N]$. Assume on the affine chart $\{x_1\neq 0\}$, the curve $C$ is defined by polynomials $g_1,\dots,g_m\in K[x_0,x_2,\dots,x_N]$. Then for $2\leq k \leq N$, we have
$$
x_k^{i(P)} = x_0\frac{r_{k0}(x_0,x_2,\dots,x_N)}{1+s_{k0}(x_0,x_2,\dots,x_N)}+\sum_{j=1}^m g_j(x_0,x_2,\dots,x_N) \frac{r_{kj}(x_0,x_2,\dots,x_N)}{1+s_{kj}(x_0,x_2,\dots,x_N)}
$$
for some polynomials $r_{kj}\in K[x_0,x_2,\dots,x_N]$ and $s_{kj}\in(x_0,x_2,\dots,x_N)K[x_0,x_2,\dots,x_N]$.

Let $c_{1,v}=\sum\limits_{2\leq k\leq N, 0\leq j\leq m} \left(l_v(s_{kj})+\log^+ N(s_{kj})\right)+\log 2$ and $c_{2,v}=\sum\limits_{2\leq k\leq N} \left(l_v(r_{k0})+\log^+ N(r_{k0})\right)+\log 2$. We prove $\lambda_{H_\infty}(v,x)\leq i(P)\lambda_P^1(v,x)+c_{2,v}$ when $\lambda_P^1(v,x)\geq c_{1,v}$.

Since $\lambda_P^1(v,x)>0$, we know $x_1\neq 0$. Then we can write $x=[\frac{1}{x_1},1,\frac{x_2}{x_1},\dots,\frac{x_N}{x_1}]$. Further, we have $\|\frac{1}{x_1}\|_v\leq 1$ and $\|\frac{x_k}{x_1}\|_v\leq 1$ for $2\leq k\leq N$.  It follows that $\lambda_{H_\infty}(v,x)=\log^+ \max\limits_{1\leq j\leq N}\|x_j\|_v=\log\|x_1\|_v$ and $\|r_{k0}(\frac{1}{x_1},\frac{x_2}{x_1},\dots,\frac{x_N}{x_1})\|_v\leq N(r_{k0})H_v(r_{k0})$. By $\lambda_P^1(v,x)\geq c_{1,v}$, we have $\|\frac{1}{x_1}\|_v\leq \frac{1}{2N(s_{kj})H_v(s_{kj})}$ and $\|\frac{x_k}{x_1}\|_v\leq \frac{1}{2N(s_{kj})H_v(s_{kj})}$ for $2\leq k\leq N$ and $0\leq j\leq m$. Hence
$$
\|s_{kj}(\frac{1}{x_1},\frac{x_2}{x_1},\dots,\frac{x_N}{x_1})\|_v\leq N(s_{kj})H_v(s_{kj})\max\{\|\frac{1}{x_1}\|_v,\|\frac{x_2}{x_1}\|_v, \dots, \|\frac{x_N}{x_1}\|_v\}\leq \frac{1}{2}.
$$
    
Since $x\in C(K)$, we have $g_j(\frac{1}{x_1},\frac{x_2}{x_1},\dots,\frac{x_N}{x_1})=0$ for $1\leq j\leq m$. It follows that
$$
\|\frac{x_k}{x_1}\|_v^{i(P)}=\|\frac{1}{x_1}\cdot\frac{r_{k0}(\frac{1}{x_1},\frac{x_2}{x_1},\dots,\frac{x_N}{x_1})}{1+s_{k0}(\frac{1}{x_1},\frac{x_2}{x_1},\dots,\frac{x_N}{x_1})}\|_v\leq\|\frac{1}{x_1}\|_v \cdot 2N(r_{k0})H_v(r_{k0}).
$$
Taking logarithm, we get 
$$
\lambda_{H_\infty}(v,x) \leq i(P)\lambda_P^1(v,x)+\log^+ \left(2\max\limits_{2\leq k\leq N} N(r_{k0})H_v(r_{k0})\right)
$$
and finish the proof.
\endproof

We omit the proof of Lemma \ref{3.16} as it is fairly easy. It amounts to saying that the hypotenuse of a right-angled triangle is longer than its legs.

\section*{Acknowledgements}
We are grateful to our advisor, Junyi Xie, for introducing this topic to us and many beneficial discussions through the work. We would like to thank Guoquan Gao, Jiarui Song and Chengyuan Yang for some useful discussion.

This work is supported by the National Natural Science Foundation of China Grant No. 12271007.

\bibliographystyle{alpha}
\bibliography{reference}

\address{Beijing International Center for Mathematical Research, Peking University, Beijing 100871, China}

\email{ys-yx@pku.edu.cn}

~

\address{Beijing International Center for Mathematical Research, Peking University, Beijing 100871, China}

\email{zay@pku.edu.cn}

\end{spacing}
\end{document}